\documentclass[11pt]{amsart}
\usepackage{amsmath,amssymb,amsthm,amsxtra,graphics,enumerate} 
\usepackage{stmaryrd}
\usepackage{times}
\usepackage[margin=1.3in]{geometry}
\usepackage[linktocpage=true, hidelinks, colorlinks=false, allcolors=blue]{hyperref}
\usepackage{comment}
\usepackage{scalerel}
\usepackage{mathrsfs}

\usepackage{comment}

\theoremstyle{definition}

\newtheorem{thm}{Theorem}[section]

\newtheorem{cor}[thm]{Corollary}

\newtheorem{lemma}[thm]{Lemma}
\newtheorem{prop}[thm]{Proposition}

\newtheorem{defn}[thm]{Definition}

\newtheorem*{maintheorem}{Main Theorem}

\newcommand\Acal{\mathcal{A}}

\newcommand\Ccal{\mathcal{C}}
\newcommand\Dcal{\mathcal{D}}
\newcommand\Ecal{\mathcal{E}}
\newcommand\Fcal{\mathcal{F}}
\newcommand\Gcal{\mathcal{G}}

\newcommand\Ical{\mathcal{I}}

\newcommand\Mcal{\mathcal{M}}

\newcommand{\Ycal}{\mathcal{Y}}

\newcommand{\B}{\mathbb{B}}
\newcommand{\C}{\mathbb{C}}

\newcommand{\p}{\mathbb{P}}
\newcommand{\q}{\mathbb{Q}}

\newcommand{\ot}{\mathrm{ot}}
\newcommand{\cf}{\mathrm{cf}}

\newcommand{\dom}{\mathrm{dom}}
\newcommand{\ran}{\mathrm{ran}}

\newcommand{\res}{\upharpoonright}

\begin{document}

\title[Combinatorial Properties Related to the Higher Baumgartner's Axiom]{Combinatorial Properties Related to the \\ Higher Baumgartner's Axiom}

\author{John Krueger}

\address{Department of Mathematics, 
	University of North Texas, Denton, TX, USA}
\email{john.krueger@unt.edu}



\date{March 21, 2026; revised September 13, 2026.}

\begin{abstract}
	We isolate two combinatorial properties, each expressible by a $\Pi_2$-sentence 
	over the structure $(H(\omega_3),\in,\omega_1,\omega_2,\text{NS}_{\omega_2})$, 
	such that each property is consistent with \textsf{CH}, 
	and their conjunction together with $2^\omega \le \omega_2$ and 
	$2^{\omega_2} = \omega_3$ implies the 
	existence of a c.c.c.\ forcing which forces the higher Baumgartner's axiom.
\end{abstract}

\maketitle

\section{Introduction} \label{Introduction}

For any uncountable cardinal $\kappa$, a set of reals $X$ is \emph{$\kappa$-dense} 
if for all $x < y$ in $X$, there exist exactly $\kappa$-many members of $X$ 
between $x$ and $y$. 
\emph{Baumgartner's axiom for $\kappa$}, or $\textsf{BA}_{\kappa}$, is the property 
that any two $\kappa$-dense sets of reals without endpoints are isomorphic. 
The statement $\textsf{BA}_{\omega_1}$, which is the only case of the above property 
which is currently known 
to be consistent, is known as \emph{Baumgartner's axiom}. 
Baumgartner \cite{baumgartner} proved that assuming \textsf{CH}, 
any two $\omega_1$-dense sets of reals 
without endpoints can be made isomorphic by a c.c.c.\ forcing of size $\omega_1$. 
Thus, starting with a model of \textsf{GCH}, there exists a finite support forcing 
iteration of c.c.c.\ forcings of length $\omega_2$ which forces that $2^\omega = \omega_2$ 
and $\textsf{BA}_{\omega_1}$ holds.

Baumgartner \cite{baumgartner2} asked whether it is consistent that any two 
$\omega_2$-dense sets of reals without endpoints are isomorphic, a statement which we 
refer to as the \emph{higher Baumgartner's axiom}. 
Despite the interest in this problem by Baumgartner and others (e.g., \cite{ARS}), 
no progress was made on it for more than three decades. 
By a result of Sierpi\'{n}ski, 
\textsf{ZFC} implies that there are at least $(2^\omega)^+$-many distinct 
order types of sets of reals of size $2^\omega$ 
(\cite{sierpinski}; also see \cite[Theorem 2.4(a)]{baumgartner2}). 
On the other hand, Abraham, Rubin, and Shelah \cite{ARS} proved the consistency of 
$\textsf{BA}_{\omega_1}$ together with $2^\omega > \omega_2$. 
Hence, a natural generalization of Baumgartner's axiom to higher cardinals is the 
statement that $2^\omega > \omega_2$ and for any uncountable cardinal $\kappa < 2^\omega$, 
$\textsf{BA}_{\kappa}$ holds.

Recently, Moore and Todor\v{c}evi\'{c} \cite{MT} made a potential breakthrough on 
Baumgartner's question. 
They introduced a combinatorial property of $\omega_2$ denoted by $(**)$ which, 
together with $2^\omega \le \omega_2$ and $2^{\omega_2} = \omega_3$, 
implies the existence of a c.c.c.\ forcing which 
forces $2^\omega  = \omega_3$, $\textsf{MA}_{\omega_2}$, $\textsf{BA}_{\omega_1}$, 
and $\textsf{BA}_{\omega_2}$. 
The general idea of their proof is to bootstrap a known c.c.c.\ forcing for adding an 
increasing function between two sets of reals of size $\omega_1$ 
to sets of reals of size $\omega_2$ 
by way of a sequence of bijections between $\omega_1$ and uncountable ordinals 
less than $\omega_2$ (\cite[Theorem 4.2]{partitionproblems}). 
They showed that assuming $(**)$ and $2^\omega \le \omega_2$, 
there is a forcing which adds an 
increasing function from a set of reals of size $\omega_2$ into another which can be 
written as an increasing union with length $\omega_2$ of c.c.c.\ suborders, 
and hence is itself c.c.c. 
Around the time that these results were announced, there were indications that a proof 
of the consistency of $(**)$ from large cardinals would be forthcoming (see 
\cite[Section 5]{MT}), but more than a decade later the status of the consistency 
of $(**)$ is still unresolved.

In this article, we isolate two statements $(A)$ and $(B)$, each consistent with \textsf{CH}, 
whose conjunction is equivalent to $(**)$. 
Both statements can be expressed by $\Pi_2$ sentences in the language of 
the structure $(H(\omega_3),\in,\omega_1,\omega_2,\text{NS}_{\omega_2})$. 
Statement $(A)$ is the assertion that any family of $\omega_2$-many club subsets 
of $\omega_1$ can be diagonalized, which is equivalent to 
$\mathfrak{b}(\omega_1) > \omega_2$. 
Statement $(B)$, defined in Definition \ref{definition of (B)} below, 
is a weakening of $(**)$ obtained by fixing a sequence of bijections between $\omega_1$
and the uncountable ordinals below $\omega_2$ and requiring the
relevant disagreement properties of $(**)$ locally on stationarily many ordinals. 
The consistency of $(A)$ is well-known and can be obtained by a countable support forcing 
iteration of the classical fast club forcing of Jensen. 
Our proof of the consistency of $(B)$ 
uses a countable support forcing iteration with a thin system of 
uncountable models as side conditions, chosen to have restrained 
mutual relationships.\footnote{The idea of using models as side conditions in a forcing iteration 
is due to Neeman and was originally used to construct a model 
of \textsf{PFA} with a finite support iteration (\cite{twotype}).}

\begin{maintheorem}
	Assuming the consistency of \textsf{ZFC}, 
	it is consistent that \textsf{CH} and (B) hold.
\end{maintheorem}

A key feature of our approach is to fix in advance a nice family of models to 
use as side conditions, where the relevant relationship between the models is a 
feature of the family. 
In contrast, many side condition forcings include arbitrary 
elementary substructures as side conditions 
with additional restrictions on how models appearing in a condition interact 
(e.g.\ \cite{mitchell}). 
Our ``off-the-shelf'' approach to side conditions, in which the family
of models is fixed in advance rather than constructed from scratch by forcing, 
provides some simplifications, since it trivializes the parts of amalgamation
arguments which purely involve models.
We also used this idea in \cite{jk31} 
in our construction of a model of \textsf{CH} 
in which any two countably closed 
$\omega_2$-Aronszajn trees are club isomorphic (\cite{jk31}). 
A major difference between \cite{jk31} and this article is that here 
we define the family of models using 
$\Box_{\omega_2}$ rather than from a large cardinal assumption. 
As a result, we prove the consistency of $(B)$ assuming only the consistency of \textsf{ZFC}.

\section{Combinatorial Properties} 
\label{Combinatorial Properties}

We begin by reviewing the property $(**)$ and proving that it is equivalent to the 
conjunction of the statements $(A)$ and $(B)$ introduced below.

\begin{defn}[Moore and Todor\v{c}evi\'{c} (\cite{MT})] \label{original **}
	Define $(**)$ to be the statement that for any family $\Fcal \subseteq {}^{\omega_2} \omega_2$ 
	of injective functions with $|\Fcal| \le \omega_2$, 
	there exists an injective 
	function $h : \omega_2 \to \omega_2$ such that for any $f \in \Fcal$, 
	there exists a countable set $D \subseteq \omega_2$ such that:
	\begin{itemize}
		\item for all $\alpha \in \omega_2 \setminus D$, 
		$f(\alpha) \ne h(\alpha)$;
		\item for all distinct $\alpha, \beta \in \omega_2 \setminus D$, 
		$f(h(\alpha)) \ne h(\beta)$.
	\end{itemize}
\end{defn}

\begin{thm}[Moore and Todor\v{c}evi\'{c} (\cite{MT})]
	Assume that $2^\omega \le \omega_2$, $2^{\omega_2} = \omega_3$, and $(**)$ holds. 
	Then there exists a c.c.c.\ forcing which forces $\textsf{MA}_{\omega_2}$, 
	$\textsf{BA}_{\omega_1}$, and $\textsf{BA}_{\omega_2}$.
\end{thm}

We isolate two statements (A) and (B), which jointly imply that $(**)$ holds.

\begin{defn}
	Let $(A)$ denote the statement that for any family $\Ccal$ of at most 
	$\omega_2$-many club subsets of $\omega_1$, there exists a club $E$ 
	such that for all $C \in \Ccal$, $E \setminus C$ is countable.
\end{defn}

The consistency of $(A)$ is due to Jensen. 
Namely, assuming \textsf{CH} there exists an $\omega_1$-closed, $\omega_2$-centered 
forcing which adds a club subset $E$ of $\omega_1$ such that for all clubs 
$C$ in the ground model, $E \setminus C$ is countable (\cite[Chapter IX]{devlinj}). 
Starting with a model of \textsf{GCH}, this fast club forcing can be iterated with 
countable support up to $\omega_3$ to obtain a generic extension satisfying (A) 
and \textsf{CH}.

The following lemma is well-known 
(see \cite[Theorem 2.5]{schilhan} for a proof).

\begin{lemma} \label{fast club}
	(A) is equivalent to $\mathfrak{b}(\omega_1) > \omega_2$. 
	More precisely, (A) is equivalent to the statement that 
	for any family $\Fcal \subseteq {}^{\omega_1} \omega_1$ of 
	size at most $\omega_2$, there exists an injective 
	function $g : \omega_1 \to \omega_1$ 
	such that for all $f \in \Fcal$, there exists $i < \omega_1$ 
	such that for all $\gamma \in \omega_1 \setminus i$, 
	$f(\gamma) < g(\gamma)$.
\end{lemma}

Define $S^2_1  = \{ \alpha < \omega_2 : \cf(\alpha) = \omega_1 \}$. 
For any partial functions $h_0,\ldots,h_{n-1} : \omega_2 \to \omega_2$, define 
$\mathrm{Cl}_{h_0,\ldots,h_{n-1}}$ 
as the set of $\beta < \omega_2$ which are closed under $h_0,\ldots,h_{n-1}$.

\begin{defn} \label{definition of (B)}
	Define $(B)$ to be the statement that there exists a sequence 
	$\vec \pi = \langle \pi_\beta : \omega_1 \le \beta < \omega_2 \rangle$, 
	where each $\pi_\beta : \beta \to \omega_1$ is a bijection, such that for any 
	injective function $f \in {}^{\omega_1} \omega_1$, there exists
	an injective function $h : \omega_2 \to \omega_2$ 
	and there exists a stationary set $I \subseteq S^2_1 \cap \mathrm{Cl}_h$ 
	such that for all $\beta \in I$, there exists 
	a countable set $D_{\beta} \subseteq \beta$ such that:
	\begin{itemize}
		\item for all $\alpha \in \beta \setminus D_{\beta}$, 
		$$
		f(\min \{ \pi_\beta(\alpha), \pi_\beta(h(\alpha)) \}) < 
		\max \{ \pi_\beta(\alpha), \pi_\beta(h(\alpha)) \};
		$$
		\item for all distinct $\alpha, \gamma \in \beta \setminus D_{\beta}$, 
		$$
		f(\min \{ \pi_\beta(h(\alpha)), \pi_\beta(h(\gamma)) \}) < 
		\max( \{ \pi_\beta(h(\alpha)), \pi_\beta(h(\gamma)) \} ).
		$$
	\end{itemize}
\end{defn}

\begin{lemma}
	The conjunction of $(A)$ and $(B)$ is equivalent to $(**)$.
\end{lemma}

\begin{proof}
	For the reverse direction, see Propositions 2.2 and 2.3 of \cite{MT}. 
	For the forward direction, assume $(A)$ and $(B)$ 
	and let	$\Fcal \subseteq {}^{\omega_2} \omega_2$ be a family 
	of injective functions with $|\Fcal| \le \omega_2$. 
	For each $f \in \Fcal$ and for each uncountable 
	$\beta \in \mathrm{Cl}_{f,f^{-1}}$, 
	define $g_{f,\beta,0}$ and $g_{f,\beta,1}$ in ${}^{\omega_1} \omega_1$ as follows. 
	For each $i < \omega_1$, define:
	\begin{itemize}
		\item $g_{f,\beta,0}(i) = \pi_\beta(f(\pi_{\beta}^{-1}(i)))$;
		\item $g_{f,\beta,1}(i) = \pi_\beta(f^{-1}(\pi_\beta^{-1}(i)))$ provided 
		that $\pi_\beta^{-1}(i) \in \ran(f)$, and otherwise $g_{f,\beta,1}(i) = 0$.
	\end{itemize}
	Now let $\Gcal$ be the set of all functions of the form 
	$g_{f,\beta,j}$, where $f \in \Fcal$, 
	$\beta \in \mathrm{Cl}_{f,f^{-1}}$ is uncountable, 
	and $j < 2$. 
	Then $\Gcal \subseteq {}^{\omega_1} \omega_1$ and $|\Gcal| \le \omega_2$.
	
	Applying $(A)$ and Lemma \ref{fast club}, 
	fix an injective function $g \in {}^{\omega_1} \omega_1$ such that for all 
	$k \in \Gcal$, there exists $i(k) < \omega_1$ such that for all 
	$\gamma \in \omega_1 \setminus i(k)$, $k(\gamma) < g(\gamma)$. 
	Now apply (B) to fix an injective function 
	$h : \omega_2 \to \omega_2$ and a stationary set 
	$I \subseteq S^2_1 \cap \mathrm{Cl}_h$ 
	such that for all $\beta \in I$, there exists 
	a countable set $D_{\beta} \subseteq \beta$ such that:
	\begin{itemize}
		\item for all $\alpha \in \beta \setminus D_{\beta}$, 
		$$
		g(\min \{ \pi_\beta(\alpha), \pi_\beta(h(\alpha)) \}) < 
		\max \{ \pi_\beta(\alpha), \pi_\beta(h(\alpha)) \};
		$$
		\item for all distinct $\alpha, \gamma \in \beta \setminus D_{\beta}$, 
		$$
		g(\min \{ \pi_\beta(h(\alpha)), \pi_\beta(h(\gamma)) \}) < 
		\max( \{ \pi_\beta(h(\alpha)), \pi_\beta(h(\gamma)) \} ).
		$$
	\end{itemize}

	Consider $f \in \Fcal$, and we verify the two required properties 
	of Definition \ref{original **}. 
	For the first property, 
	suppose for a contradiction that the set 
	$X = \{ \alpha < \omega_2 : f(\alpha) = h(\alpha) \}$ is uncountable. 
	By the stationarity of $I$, we can find an ordinal 
	$\beta \in I \cap \mathrm{Cl}_{f,f^{-1},h}$ 
	such that $X \cap \beta$ is uncountable. 
	Then $g_{f,\beta,0}$ and $g_{f,\beta,1}$ are in $\Gcal$. 
	Since $D_{\beta}$ is countable, 
	we can fix $\alpha \in (X \cap \beta)$ which is not in $D_\beta$ and 
	satisfies that $\pi_\beta(\alpha) \ge i(g_{f,\beta,0})$ and 
	$\pi_\beta(h(\alpha)) \ge i(g_{f,\beta,1})$. 
	As $\alpha \in X$, $f(\alpha) = h(\alpha)$. 

	Case 1: $\pi_\beta(\alpha) < \pi_\beta(h(\alpha))$. 
	Since $\alpha \in \beta \setminus D_{\beta}$, 
	we have that $g(\pi_\beta(\alpha)) < \pi_\beta(h(\alpha))$. 
	But also $\pi_\beta(\alpha) \ge i(g_{f,\beta,0})$, so 
	$g_{f,\beta,0}(\pi_\beta(\alpha)) < \pi_\beta(h(\alpha))$. 
	By the definition of $g_{f,\beta,0}$, this inequality is equivalent to 
	$\pi_\beta(f(\alpha)) < \pi_\beta(h(\alpha))$, which contradicts that 
	$f(\alpha) = h(\alpha)$.

	Case 2: $\pi_\beta(h(\alpha)) < \pi_\beta(\alpha)$. 
	Since $\alpha \in \beta \setminus D_{\beta}$, 
	we have that $g(\pi_\beta(h(\alpha))) < \pi_\beta(\alpha)$. 
	As $\pi_\beta(h(\alpha)) \ge i(g_{f,\beta,1})$, 
	$g_{f,\beta,1}(\pi_\beta(h(\alpha))) < \pi_\beta(\alpha)$. 
	Letting $i = \pi_\beta(h(\alpha))$, note that 
	$\pi_\beta^{-1}(i) = h(\alpha) = f(\alpha)$ which is in the range of $f$. 
	So by the definition of $g_{f,\beta,1}$, 
	$g_{f,\beta,1}(\pi_\beta(h(\alpha))) = 
	\pi_\beta(f^{-1}(\pi_\beta^{-1}(i))) = \pi_\beta(f^{-1}(h(\alpha))) = 
	\pi_\beta(f^{-1}(f(\alpha))) = \pi_\beta(\alpha)$. 
	Hence, the above inequality is the same as 
	$\pi_\beta(\alpha) < \pi_\beta(\alpha)$, which is a contradiction.

	For the second property, 
	suppose for a contradiction that it is not the case that there exists a countable subset 
	of $\omega_2$ such that 
	for all distinct $\alpha, \gamma \in \omega_2$ which are not in this countable set, 
	$f(h(\alpha)) \ne h(\gamma)$. 
	Then it is easy to build a set $Y \subseteq \omega_2$ of size $\omega_1$ 
	such that for all countable $Y_0 \subseteq Y$, there are distinct 
	$\alpha, \gamma \in Y \setminus Y_0$ such that 
	$f(h(\alpha)) = h(\gamma)$. 
	By the stationarity of $I$, we can 
	fix $\beta \in I \cap \mathrm{Cl}_{f,f^{-1},h}$. 
	Fix distinct $\alpha, \gamma \in \beta \setminus D_\beta$ such that 
	$f(h(\alpha)) = h(\gamma)$, $\pi_\beta(h(\alpha)) > i(g_{f,\beta,0})$, and 
	$\pi_\beta(h(\gamma)) > i(g_{f,\beta,1})$. 
	Now the rest of the proof is similar to the above.
\end{proof}

\section{The Single Forcing} \label{The Single Forcing}

The rest of the article is devoted to proving the consistency of $(B)$. 
In this section, we describe how to force a single instance of $(B)$. 
For the remainder of the section, fix:
\begin{itemize}
	\item a sequence $\vec \pi = \langle \pi_\beta : \omega_1 \le \beta < \omega_2 \rangle$, 
	where each $\pi_\beta : \beta \to \omega_1$ is a bijection;
	\item an injective function $f : \omega_1 \to \omega_1$.
\end{itemize}
Let $(B)_{\vec \pi,f}$ denote the statement that there exists 
an injective function $h : \omega_2 \to \omega_2$ 
and there exists a stationary set $I \subseteq S^2_1 \cap \mathrm{Cl}_h$ 
such that the conclusion of Definition \ref{definition of (B)} is satisfied.

\begin{defn} \label{definition of the main forcing}
	Define $\B(\vec \pi,f)$ to be the forcing whose conditions are 
	ordered pairs $p = (h_p,I_p)$ satisfying:
	\begin{enumerate}
		\item $h_p$ is an injective function whose domain is a countable 
		subset of $\omega_2$ which maps into $\omega_2$ and satisfies 
		that $h_p(\alpha) \ne \alpha$ for all $\alpha \in \dom(h_p)$;
		\item $I_p$ is a countable subset of $S^2_1$;
		\item for all $\beta \in I_p$, $\beta$ is closed under 
		$h_p$ and $h_p^{-1}$.
	\end{enumerate}
	Let $q \le_{\B(\vec \pi,f)} p$ if:
	\begin{enumerate}
		\item[(a)] $h_p \subseteq h_q$;
		\item[(b)] $I_p \subseteq I_q$;
		\item[(c)] for all $\alpha \in \dom(h_q) \setminus \dom(h_p)$ 
		and for all $\beta \in I_p$, 
		if $\alpha < \beta$ then:
		\begin{enumerate}
			\item[(i)] 
			$f(\min\{\pi_\beta(\alpha),\pi_\beta(h_q(\alpha))\}) < 
			\max\{\pi_\beta(\alpha),\pi_\beta(h_q(\alpha))\}$;
			\item[(ii)] for all $\gamma \in \dom(h_q) \cap \beta$ different from $\alpha$, 
			$$
			f(\min \{ \pi_\beta(h_q(\gamma)),\pi_\beta(h_q(\alpha)) \}) < 
			\max \{ \pi_\beta(h_q(\gamma)),\pi_\beta(h_q(\alpha)) \}.\footnote{Note that 
			the parameter $f$ is only relevant in the order on 
			$\B(\vec \pi,f)$, and not for membership in $\B(\vec \pi,f)$.}
			$$
		\end{enumerate}
	\end{enumerate}
\end{defn}

\begin{lemma}
	The relation $\le_{\B(\vec \pi,f)}$ is transitive.
\end{lemma}

\begin{proof}
	Assume that $r \le q \le p$, and we show that $r \le p$. 
	(a, b) Clearly, $h_p \subseteq h_r$ and $I_p \subseteq I_r$. 
	(c) Suppose that $\alpha \in \dom(h_r) \setminus \dom(h_p)$, 
	$\beta \in I_p$, and $\alpha < \beta$. 
	(i) If $\alpha \in \dom(h_q)$, then $q \le p$ implies (i). 
	Suppose that $\alpha \notin \dom(h_q)$. 
	Then $r \le q$ implies (i).
	(ii) Let $\gamma \in \dom(h_r) \cap \beta$ be different from $\alpha$. 
	If $\alpha \notin \dom(h_q)$, then 
	the conclusion of (ii) follows from $r \le q$. 
	Suppose that $\alpha \in \dom(h_q)$. 
	If $\gamma \in \dom(h_q)$, then $q \le p$ implies the conclusion of (ii). 
	Suppose that $\gamma \notin \dom(h_q)$. 
	Then the conclusion of (ii) follows from $r \le q$ with the roles of $\alpha$ 
	and $\gamma$ reversed.
\end{proof}

Going forward, we abbreviate $\le_{\B(\vec \pi,f)}$ by $\le$.

\begin{lemma} \label{K is omega_1 closed}
	The forcing $\B(\vec \pi,f)$ is $\omega_1$-closed. 
	In fact, if $\langle p_n : n < \omega \rangle$ is a descending sequence of conditions 
	in $\B(\vec \pi,f)$, where each $p_n = (h_n,I_n)$, then 
	$(\bigcup_n h_n,\bigcup_n I_n)$ is a condition in $\B(\vec \pi,f)$ 
	which is the greatest lower bound of $\langle p_n : n < \omega \rangle$.
\end{lemma}

The proof is easy.

\begin{cor}
	The forcing $\B(\vec \pi,f)$ preserves $\omega_1$.
\end{cor}

Whether and under what circumstances 
$\B(\vec \pi,f)$ preserves $\omega_2$ is a complex issue which 
we address over the remainder of the article.

\begin{lemma} \label{h is total}
	For any countable set $Y \subseteq \omega_2$, 
	the set of $q \in \B(\vec \pi,f)$ such that $Y \subseteq \dom(h_q)$ 
	is dense open in $\B(\vec \pi,f)$.
\end{lemma}

\begin{proof}
	By Lemma \ref{K is omega_1 closed}, 
	it suffices to prove the statement in the case that $Y$ is a singleton.  
	Let $p \in \B(\vec \pi,f)$ and let $\alpha \in \omega_2$. 
	If $\alpha \in \dom(h_p)$ then we are done, so assume not. 
	We find $q \le p$ such that $\alpha \in \dom(h_q)$. 
	For each $\xi \in I_p$ greater than $\alpha$, 
	define $X_{\xi}$ as the set of $\nu < \xi$ such that at 
	least one of the following is true:
	\begin{enumerate}
		\item $\pi_\xi(\nu) \le \pi_\xi(\alpha)$;
		\item $\pi_\xi(\nu) \le f(\pi_\xi(\alpha))$;
		\item there exists $\gamma \in \dom(h_p) \cap \xi$ such that 
		$\pi_\xi(\nu) \le \pi_\xi(h_p(\gamma))$;
		\item there exists $\gamma \in \dom(h_p) \cap \xi$ such that 
		$\pi_\xi(\nu) \le f(\pi_\xi(h_p(\gamma)))$.
	\end{enumerate}
	Since $\pi_\xi$ is injective and $\dom(h_p)$ is countable, 
	$X_{\xi}$ is countable. 
	Let $X = \bigcup \{ X_{\xi} : \xi \in I_p, \ \xi > \alpha \}$, 
	which is countable.  
	Let $\beta$ be the least member of $I_p \cup \{ \omega_2 \}$ 
	which is strictly greater than $\alpha$. 
	Then $\beta$ has cofinality at least $\omega_1$, 
	so we can find some $\zeta \in (\alpha,\beta)$ 
	which is not in the range of $h_p$ nor in $X$. 
	Define $q$ by letting $h_q = h_p \cup \{ \langle \alpha, \zeta \rangle \}$  
	and $I_q = I_p$. 
	It is straightforward to check that $q$ is as required.
\end{proof}

\begin{lemma} \label{K forces (B)} 
	Suppose that $\B(\vec \pi,f)$ preserves $\omega_2$ and forces that 
	$\dot \Ical = \bigcup \{ I_p : p \in \dot G \}$ is stationary in $\omega_2$, 
	where $\dot G$ is the 
	canonical $\B(\vec \pi,f)$-name for the generic filter. 
	Then $\B(\vec \pi,f)$ forces that 
	$(B)_{\vec \pi,f}$ holds as witnessed by $\dot \Ical$.
\end{lemma}

\begin{proof}
	Let $G$ be any generic filter on $\B(\vec \pi,f)$. 
	Define
	$$
	h = \bigcup \{ h_p : p \in G \} \ \text{and} \ 
	I = \bigcup \{ I_p : p \in G \}.
	$$
	Using Lemma \ref{h is total} together with 
	Definition \ref{definition of the main forcing}, 
	$h$ is a total injective function from $\omega_2$ to $\omega_2$, 
	$I \subseteq S^2_1$ is stationary in $\omega_2$, 
	and for all $\beta \in I$ and for all $\alpha < \beta$, 
	$h(\alpha) < \beta$.

	Consider $\beta \in I$. 	
	Fix $p \in G$ such that $\beta \in I_p$. 
	Let $D_\beta = \dom(h_p) \cap \beta$. 
	Now for any distinct $\alpha, \gamma \in \beta \setminus D_\beta$, 
	there exists some $q \le p$ in $G$ 
	such that $\alpha, \gamma \in \dom(h_q) \setminus \dom(h_p)$. 
	By the definition of the order on $\B(\vec \pi,f)$ and the fact 
	that $h(\alpha) = h_q(\alpha)$ and $h(\gamma) = h_q(\gamma)$, we have that 
	$$
		f(\min \{ \pi_\beta(\alpha), \pi_\beta(h(\alpha)) \}) < 
		\max \{ \pi_\beta(\alpha), \pi_\beta(h(\alpha)) \}
	$$
	and 
	$$
		f(\min \{ \pi_\beta(h(\alpha)), \pi_\beta(h(\gamma)) \}) < 
		\max( \{ \pi_\beta(h(\alpha)), \pi_\beta(h(\gamma)) \} ).
	$$
\end{proof}

\section{Adding the Bijections} \label{Adding the Bijections}

One component of our strategy for preserving $\omega_2$ after forcing 
with $\B(\vec \pi,f)$ is to obtain the sequence of bijections 
$\vec \pi$ generically by forcing with countable conditions. 

\begin{defn}
	Define a forcing $\C$ to consist of conditions which are 
	functions $p$ with domain a countable subset of 
	$\bigcup \{ \{ \beta \} \times \beta : \omega_1 \le \beta < \omega_2 \}$ 
	such that for all $\omega_1 \le \beta < \omega_2$, 
	$p \res \{ \beta \} \times \beta$ is an injective function mapping into $\omega_1$. 
	Let $q \le p$ in $\C$ if $p \subseteq q$.
\end{defn}

The proofs of the next three lemmas are routine.

\begin{lemma} \label{C is omega_1 closed}
	The forcing $\C$ is $\omega_1$-closed. 
	In fact, if $\langle c_n : n < \omega \rangle$ is a descending sequence 
	of conditions in $\C$, 
	then $\bigcup_n c_n$ is in $\C$ 
	and is the greatest lower bound of $\langle c_n : n < \omega \rangle$.	
\end{lemma}

\begin{lemma}
	\textsf{CH} implies that $\C$ is $\omega_2$-c.c.
\end{lemma}

\begin{lemma} \label{generic bijections}
	Suppose that $G$ is a generic filter on $\C$. 
	For each $\omega_1 \le \beta < \omega_2$, define 
	$\pi_\beta : \beta \to \omega_1$ by letting $\pi_\beta(\gamma) = c(\beta,\gamma)$ 
	for some (any) $c \in G$ such that $(\beta,\gamma) \in \dom(c)$. 
	Then each $\pi_\beta$ is a bijection of $\beta$ onto $\omega_1$.
\end{lemma}

\section{The Definition of the Forcing Iteration} \label{The Definition of the Forcing Iteration}

In order to iterate forcings of the form $\B(\vec \pi,f)$ while preserving $\omega_2$, 
we include side conditions in our forcing iteration conditions. 
For that purpose, we make use of a 
well-behaved family of models for which we can construct generic conditions. 
Such a family is described in the next theorem.

\begin{thm}[Essentially Mitchell (\cite{mitchell})] \label{Ycal}
	Assume $\textsf{CH}$, $2^{\omega_2} = \omega_3$, and $\Box_{\omega_2}$. 
	Then there exist stationary subsets $\Ycal$ and $\Ycal^+$ of $[H(\omega_3)]^{\omega_1}$ 
	satisfying:
	\begin{enumerate}
		\item for all $M \in \Ycal$, $M \prec H(\omega_3)$, $|M| = \omega_1$, 
		and $M^\omega \subseteq M$;
		\item for all $M, N \in \Ycal$, $M \cap N \in \Ycal$;
		\item $\Ycal^+ \subseteq \Ycal$;
		\item for all $M \in \Ycal$ and for all $N \in \Ycal^+$, 
		if $M \cap \omega_2 < N \cap \omega_2$ then 
		$M \cap N \in N$.\footnote{The reason we need two families of models 
		instead of one is because 
		$\Ycal^+$ is not closed under intersections.}
	\end{enumerate}
\end{thm}

While not stated in this exact manner in Mitchell's work, 
this result constitutes a small fragment of 
the information contained in his proof of the consistency that the approachability ideal 
$I[\omega_2]$ restricted to $\text{Cof}(\omega_1)$ can be trivial (\cite{mitchell}). 
We include a proof for the convenience of the reader in the appendix.

For the remainder of the article, assume that \textsf{GCH} holds and there exist 
families $\Ycal$ and $\Ycal^+$ as described in Theorem \ref{Ycal}. 
For example, these statements hold if $V = L$.

We define by recursion a sequence of posets 
$\langle \p_\delta : \delta \le \Delta \rangle$, where $\Delta \le \omega_3$ 
is some ordinal whose value is to be determined. 
Our goal is to show that the recursion does not terminate, and in that case, 
$\Delta = \omega_3$ and $\p_{\omega_3}$ preserves all cardinals. 
Until we finish our work, we cannot conclude that the recursion does not 
fail at some successor ordinal $\delta+1$, where $\delta < \omega_3$, 
and if so, then $\Delta = \delta$. 
Roughly speaking, the forcing iteration begins with the forcing $\C$ to 
produce a generic sequence $\vec \pi$ of bijections, and then 
iterates forcings of the form $\B(\vec \pi,f)$, bookkeeping to handle 
every injective function $f$ from $\omega_1$ to $\omega_1$.

Assuming that the recursion succeeds, we achieve our objective as follows. 
Each $\p_\delta$ has size at most $\omega_3$ and is $\omega_3$-c.c., and 
$\p_{\omega_3} = \bigcup \{ \p_\delta : \delta < \omega_3 \}$. 
Consequently, any nice $\p_{\omega_3}$-name for a subset of $\omega_1$ 
is a nice $\p_\delta$-name for some $\delta < \omega_3$. 
By straightforward bookkeeping, we can arrange that for every injective function 
$f$ from $\omega_1$ to $\omega_1$ in the final generic extension, there is some 
stage of the iteration at which we forced with $\B(\vec \pi,f)$. 
We also need to show that each $\B(\vec \pi,f)$ satisfies the assumptions of 
Lemma \ref{K forces (B)}, and it then follows that 
$\p_{\omega_3}$ forces that $(B)$ holds.

We use the following natural abbreviations associated with these posets, for 
each $\delta \le \Delta$:
\begin{itemize}
	\item the order on $\p_\delta$ is written as $\le_\delta$;
	\item we write $\dot G_\delta$ for the canonical $\p_\delta$-name for a generic filter 
	on $\p_\delta$;
	\item we write the forcing relation for $\p_\delta$ as $\Vdash_\delta$.
\end{itemize}

Alongside the construction of the sequence 
$\langle \p_\delta : \delta \le \Delta \rangle$, we also specify sequences 
$\langle \dot f_\tau : \tau < \Delta \rangle$ and 
$\langle \dot \q_\tau : \tau < \Delta \rangle$.

\bigskip

\textbf{Base case:}

\bigskip

Let $\p_0$ be the forcing whose conditions are ordered triples 
$(c,\emptyset,\emptyset)$, where $c \in \C$. 
Define $(d,\emptyset,\emptyset) \le_0 (c,\emptyset,\emptyset)$ 
if $c \subseteq d$. 
Note that $\p_0$ is isomorphic to $\C$.

\bigskip

Having defined $\p_0$, we fix some notation. 
Suppose that $G$ is a generic filter on $\p_0$. 
For each $\omega_1 \le \beta < \omega_2$, define 
$\pi_\beta : \beta \to \omega_1$ by letting 
$\pi_\beta(\gamma) = c(\beta,\gamma)$ for some (any) 
$(c,\emptyset,\emptyset) \in G$ such that $(\beta,\gamma) \in \dom(c)$. 
By Lemma \ref{generic bijections}, each $\pi_\beta$ is a bijection 
of $\beta$ onto $\omega_1$. 
We write $\dot \pi_\beta$ as a $\p_0$-name for this function for each 
$\omega_1 \le \beta < \omega_2$, 
and let $\dot{\vec \pi}$ be a $\p_0$-name for the sequence 
$\langle \dot \pi_\beta : \omega_1 \le \beta < \omega_2 \rangle$.

\bigskip

\textbf{Successor case:}

\bigskip

Assume that $\delta < \omega_3$ 
and $\p_\xi$ is defined for all $\xi \le \delta$. 
For the recursion to continue, we assume inductively that 
$\p_0$ is a regular suborder of $\p_\delta$ and 
$\p_\delta$ preserves all cardinals. 
If these assumptions fail, then the recursion ends and $\p_{\delta+1}$ is not defined. 
Using standard bookkeeping, 
fix a $\p_\delta$-name $\dot f_\delta$ for an injective 
function from $\omega_1$ to $\omega_1$ and let 
$\dot \q_\delta$ be a $\p_\delta$-name for 
the forcing $\B(\dot{\vec \pi},\dot f_\delta)$.

A condition in $\p_{\delta+1}$ is any ordered triple $p = (c_p,s_p,A_p)$ satisfying:
\begin{enumerate}
	\item $s_p$ is a function whose domain is a countable subset of $\delta+1$;
	\item $A_p$ is a countable set of pairs of the form $(\xi,M)$, where 
	$\xi \le \delta$, $\xi$ is a limit ordinal, and $M \in \Ycal$;
	\item the ordered triple 
	$$
	p \res \delta = 
	( c_p, s_p \res \delta, A_p )
	$$
	is a member of $\p_\delta$;\footnote{For clarity, we remark that we do not need 
	to restrict $A_p$ to $\delta$ since by (2), $(\xi,M) \in A_p$ implies that 
	$\xi < \delta+1$.}
	\item if $\delta \in \dom(s_p)$, then 
	$s_p(\delta)$ is an ordered pair $(h_{p,\delta},I_{p,\delta})$ such that:
		\begin{enumerate}
			\item $h_{p,\delta}$ is an injective function whose domain is a countable 
			subset of $\omega_2$ which maps into $\omega_2$ and 
			satisfies that $h_{p,\delta}(\alpha) \ne \alpha$ for all 
			$\alpha \in \dom(h_{p,\delta})$;
			\item $I_{p,\delta}$ is a countable subset of $S^2_1$;
			\item for all $\beta \in I_{p,\delta}$, $\beta$ is closed 
			under $h_{p,\delta}$ and $h_{p,\delta}^{-1}$.
		\end{enumerate}
\end{enumerate}
Let $q \le_{\delta+1} p$ if:
\begin{enumerate}
	\item[(a)] $A_p \subseteq A_q$;
	\item[(b)] $q \res \delta \le_{\delta} p \res \delta$;
	\item[(c)] if $\delta \in \dom(s_p)$, then $\delta \in \dom(s_q)$ 
	and $q \res \delta \Vdash_{\delta} 
	s_q(\delta) \le_{\dot \q_\delta} s_p(\delta)$.
\end{enumerate}

Note that for any $p \in \p_{\delta+1}$ with $\delta \in \dom(s_p)$, 
$\Vdash_\delta s_p(\delta) \in \B(\dot{\vec \pi},\dot f_\delta)$.
			
\bigskip

\textbf{Limit case:}

\bigskip

Assume that $\delta \le \omega_3$ is a limit ordinal 
and $\p_\xi$ is defined for all $\xi < \delta$. 
Since the recursion has continued through all ordinals less than $\delta$, 
it follows that for every $\xi < \delta$, $\p_\xi$ preserves all cardinals.

Define $\p_{\delta}$ as follows. 
A condition in $\p_\delta$ is any ordered triple $p = (c_p,s_p,A_p)$ satisfying:
\begin{enumerate}
	\item $A_p$ is a countable set of pairs of the form $(\sigma,M)$, where 
	$\sigma \le \delta$, $\sigma < \omega_3$, $\sigma$ is a limit ordinal, 
	and $M \in \Ycal$;
	\item $s_p$ is a function whose domain is a countable subset of $\delta$;
	\item for all $\xi < \delta$, the ordered triple 
	$$
	p \res \xi = 
	(c_p, s_p \res \xi, \{ (\sigma,M) \in A_p : \sigma \le \xi \})
	$$
	is a member of $\p_\xi$;
	\item for all $(\delta,M) \in A_p$ and for all 
	$\tau \in M \cap \dom(s_p) \cap \delta$, 
	$M \cap \omega_2 \in I_{p,\tau}$.
\end{enumerate}
Let $q \le_\delta p$ if:
\begin{enumerate}
	\item[(a)] $A_p \subseteq A_q$;
	\item[(b)] for all $\xi < \delta$, $q \res \xi \le_{\xi} p \res \xi$.
\end{enumerate}

\bigskip

This completes the definition of the forcing iteration.

\begin{defn}
	Let $\Delta$ be the least ordinal $\delta \le \omega_3$ such that 
	either $\delta = \omega_3$, or else $\delta < \omega_3$, $\p_\delta$ is defined, and 
	$\p_{\delta+1}$ is not defined.
\end{defn}

\begin{defn}
	Let $\xi < \delta \le \Delta$. 
	For any $p = (c_p,s_p,A_p)$ in $\p_\delta$, 
	define 
	$$
	p \res \xi = (c_p,s_p \res \xi,\{ (\sigma,M) \in A_p : \sigma \le \xi \}).
	$$
\end{defn}

The above notation suggests that the defined function is a reduction mapping 
from $\p_\delta$ to $\p_\xi$, however that is only true when restricting to a 
dense subset of $\p_\delta$. 
See Definition \ref{dense for reduction} and Lemmas \ref{dense set} and \ref{projection} 
below for the details. 

We give a useful non-inductive characterization of the forcing iteration, which can be 
proved with a routine argument by induction. 
Going forward 
we always refer to this lemma to verify membership in $\p_\delta$ and being 
related by $\le_\delta$, rather 
than using the recursive definition.

\begin{lemma} \label{non-inductive characterization}
	Let $\delta \le \Delta$. 
	Then $p \in \p_\delta$ iff $p = (c_p,s_p,A_p)$ is an ordered triple satisfying:
	\begin{enumerate}
		\item $c_p \in \C$;
		\item $s_p$ is a function whose domain is a countable subset of $\delta$;
		\item for all $\tau \in \dom(s_p)$, 
		$s_p(\tau)$ is an ordered pair $(h_{p,\tau},I_{p,\tau})$ such that: 
			\begin{enumerate}
			\item $h_{p,\tau}$ is an injective function whose domain is a countable 
			subset of $\omega_2$ which maps into $\omega_2$ and satisfies that 
			$h_{p,\tau}(\alpha) \ne \alpha$ for all $\alpha \in \dom(h_{p,\tau})$;
			\item $I_{p,\tau}$ is a countable subset of $S^2_1$;
			\item for all $\beta \in I_{p,\tau}$, $\beta$ is closed under 
			$h_{p,\tau}$ and $h_{p,\tau}^{-1}$;
		\end{enumerate}
		\item $A_p$ is a countable set of pairs of the form $(\xi,M)$, 
		where $\xi \le \delta$, $\xi < \omega_3$, $\xi$ is a limit ordinal, 
		and $M \in \Ycal$;
		\item for all $(\xi,M) \in A_p$ and for all 
		$\tau \in M \cap \dom(s_p) \cap \xi$, 
		$M \cap \omega_2 \in I_{p,\tau}$.
	\end{enumerate}
	For $p, q \in \p_\delta$, 
	$q \le_\delta p$ iff:
	\begin{enumerate}
		\item[(a)] $c_p \subseteq c_q$;
		\item[(b)] $A_p \subseteq A_q$;
		\item[(c)] $\dom(s_p) \subseteq \dom(s_q)$;
		\item[(d)] for all $\tau \in \dom(s_p)$, 
		$q \res \tau \Vdash_\tau s_q(\tau) \le_{\dot \q_\tau} s_p(\tau)$.
	\end{enumerate}
\end{lemma}

\begin{defn}
	Let $\delta \le \Delta$. 
	For any $p \in \p_\delta$ and for any $\tau \in \dom(s_p)$, 
	write $s_p(\tau) = (h_{p,\tau},I_{p,\tau})$.
\end{defn}

\section{Basic Facts About the Iteration} \label{Basic Facts About the Iteration}

In this section, we establish some basic information about the forcing iteration 
just introduced.

\begin{lemma} \label{iteration is omega_1 closed}
	Let $\delta \le \Delta$. 
	Then $\p_\delta$ is $\omega_1$-closed. 
	In fact, suppose that $\langle p_n : n < \omega \rangle$ is a descending sequence of 
	conditions in $\p_\delta$. 
	Define $q = (c_q,s_q,A_q)$, where:
	\begin{itemize}
		\item $c_q = \bigcup_n c_{p_n}$;
		\item $A_q = \bigcup_n A_{p_n}$;
		\item $\dom(s_q) = \bigcup_n \dom(s_{p_n})$;
		\item for all $\tau \in \dom(s_q)$, 
		$s_q(\tau) = (h_{q,\tau},I_{q,\tau})$, where 
		$$
		h_{q,\tau} = \bigcup \{ h_{p_n,\tau} : n < \omega, \ \tau \in \dom(s_{p_n}) \}
		$$
		and 
		$$
		I_{q,\tau} = \bigcup \{ I_{p_n,\tau} : n < \omega, \ \tau \in \dom(s_{p_n}) \}.
		$$
	\end{itemize}
	Then $q \in \p_\delta$ and $q$ is the greatest lower bound of 
	$\langle p_n : n < \omega \rangle$.
\end{lemma}

\begin{proof}
The proof is straightforward using Lemmas 
\ref{K is omega_1 closed} and \ref{C is omega_1 closed}.
\end{proof}

\begin{lemma} \label{domain dense}
	Let $\tau < \delta \le \Delta$. 
	The set of $q \in \p_\delta$ such that $\tau \in \dom(s_q)$ 
	is dense in $\p_\delta$.
\end{lemma}

\begin{proof}
	Let $p \in \p_\delta$, and we find $q \le_\delta p$ with $\tau \in \dom(s_q)$. 
	If $\tau \in \dom(s_p)$, then we are done, so assume not. 
	Define $q = (c_q,s_q,A_q)$, where $c_q = c_p$, $A_q = A_p$, 
	$\dom(s_q) = \dom(s_p) \cup \{ \tau \}$, 
	$s_q \res \dom(s_p) = s_p$, and 
	$$
	s_q(\tau) = (\emptyset,\{ M \cap \omega_2 : \exists \xi \ 
	(\xi,M) \in A_p, \ \tau \in M \cap \xi \}).
	$$
	It easily follows by Lemma \ref{non-inductive characterization} that 
	$q \in \p_\delta$ and $q \le_\delta p$.
\end{proof}

\begin{defn}
	For any $\zeta \le \Delta$, define $\zeta^{-}$ as follows:
	\begin{itemize}
		\item if $\zeta \ge \omega$, then $\zeta^{-}$ is the largest limit 
		ordinal less than or equal to $\zeta$;
		\item if $\zeta < \omega$, then $\zeta^{-} = 0$.
	\end{itemize}
\end{defn}

\begin{defn} \label{dense for reduction}
	Let $\omega \le \zeta < \delta \le \Delta$. 
	Define $D_{\delta,\zeta}$ as the set of $p \in \p_\delta$ such that: 
	\begin{itemize}
		\item for all $\zeta^{-} \le \sigma < \zeta$, $\sigma \in \dom(s_p)$;
		\item for any $(\xi,M) \in A_p$ with $\zeta < \xi \le \delta$, 
		$(\zeta^{-},M) \in A_p$.
	\end{itemize}
\end{defn}

\begin{lemma} \label{dense set}
	Let $\omega \le \zeta < \delta \le \Delta$. 
	Then $D_{\delta,\zeta}$ is dense in $\p_\delta$.
\end{lemma}

\begin{proof}
	It is easy to show that whenever $p \in \p_\delta$, 
	$\gamma < \xi \le \delta \le \Delta$, $\gamma$ is a limit ordinal, and $(\xi,M) \in A_p$, 
	then $(c_p,s_p,A_p \cup \{ (\gamma,M) \})$ is in $\p_\delta$ and extends $p$. 
	Now we are done by 
	Lemmas \ref{domain dense} and \ref{iteration is omega_1 closed}.
\end{proof}

\begin{lemma} \label{projection}
	Let $\zeta < \delta \le \Delta$. 
	Let $p \in \p_\delta$, and in the case that $\zeta \ge \omega$, also assume 
	that $p \in D_{\delta,\zeta}$. 
	Suppose that $u \le_\zeta p \res \zeta$. 
	Then $p$ and $u$ are compatible in $\p_\delta$. 
	More specifically, define $q = (c_q,s_q,A_q)$ where 
	$c_q = c_u$, $s_q = s_u \cup s_p \res [\zeta,\delta)$, 
	and $A_q = A_u \cup A_p$. 
	Then $q$ is in $\p_\delta$ and extends $u$ and $p$.	
\end{lemma}

\begin{proof}
	The only non-trivial thing to check is that $q$ satisfies 
	property (5) of Lemma \ref{non-inductive characterization} 
	in the case that $\omega \le \zeta$. 
	Let $(\xi,M) \in A_q$ and let $\tau \in M \cap \dom(s_q) \cap \xi$, 
	and we show that $M \cap \omega_2 \in I_{q,\tau}$. 
	If $(\xi,M) \in A_u$, then $\xi \le \zeta$ and hence 
	$\tau \in \dom(s_u)$, so we are done since $u$ is a condition. 
	Suppose that $(\xi,M) \in A_p \setminus A_u$, and in particular, $\xi > \zeta$. 
	If $\tau \in \dom(s_p)$, then we are done since $p$ is a condition. 
	Assume that $\tau \in \dom(s_u) \setminus \dom(s_p)$. 
	Then $\tau < \zeta^{-}$, for otherwise $\tau \in \dom(s_p)$ 
	since $p \in D_{\delta,\zeta}$. 
	Since $p \in D_{\delta,\zeta}$, $(\zeta^{-},M) \in A_p$. 
	So $(\zeta^{-},M) \in A_{p \res \zeta} \subseteq A_u$. 
	As $\tau \in M \cap \dom(s_u) \cap \zeta^{-}$, 
	$M \cap \omega_2 \in I_{u,\tau} = I_{q,\tau}$. 
\end{proof}

\begin{lemma} \label{regular suborder}
	Let $\zeta < \delta \le \Delta$. 
	Then $\p_{\zeta}$ is a regular suborder of $\p_\delta$.
\end{lemma}

\begin{proof}
	Using Lemma \ref{non-inductive characterization}, 
	it is straightforward to show the following:
	\begin{itemize}
		\item $\p_\zeta \subseteq \p_\delta$;
		\item $\le_\zeta$ equals $\le_\delta \cap \ \p_\zeta^2$;
		\item for all $p, q \in \p_\zeta$ and for all $r \in \p_\delta$, 
		$r \le_\delta p, q$ implies $r \res \zeta \le_\zeta p, q$.
	\end{itemize}
	Suppose that $A$ is a maximal antichain of $\p_\zeta$. 
	Then the third statement above implies that $A$ is an antichain of $\p_\delta$, 
	and a routine argument using 
	Lemma \ref{projection} shows that $A$ is a maximal antichain of $\p_\delta$.
\end{proof}

In particular, $\p_0$ is a regular suborder of $\p_\delta$ for all $\delta \le \Delta$.

\begin{lemma} \label{two step equivalence}
	Let $\delta < \Delta$. 
	Define $\Dcal$ as the set of $q \in \p_{\delta+1}$ such that $\delta \in \dom(s_q)$ 
	and define $\Ecal$ as the set of $p * \dot a$ in 
	$\p_\delta * \B(\dot{\vec \pi},\dot f_\delta)$ such that 
	for some $h$ and $I$, $p \Vdash_\delta \dot a = (\check h,\check I)$. 
	Then $\Dcal$ and $\Ecal$ are dense in $\p_{\delta+1}$ and 
	$\p_\delta * \B(\dot{\vec \pi},\dot f_\delta)$ respectively. 
	Moreover, letting $F : \Dcal \to \Ecal$ be the function defined by 
	$F(q) = (q \res \delta) * \dot a$, where 
	$\dot a$ is a $\p_\delta$-name for 
	$(h_{q,\delta},I_{q,\delta})$, then $F$ is a dense embedding. 
\end{lemma}

The proof is straightforward.

\begin{prop} \label{forcing (B)}
	Let $\tau < \delta \le \Delta$. 	
	Assume that $\p_{\delta}$ preserves $\omega_2$ and forces 
	that the set $\bigcup \{ I_{p,\tau} : p \in \dot G_{\delta}, \ \tau \in \dom(s_p) \}$ 
	is stationary in $\omega_2$, where $\dot G_{\delta}$ is the canonical $\p_{\delta}$-name 
	for the generic filter. 
	Then $\p_{\delta}$ forces $(B)_{\dot{\vec \pi},\dot f_\tau}$.
\end{prop}

\begin{proof}
	Suppose that $G$ is a generic filter on $\p_\delta$. 
	Let $\vec \pi = (\dot{\vec \pi})^{G}$ and $f_\tau = \dot f_\tau^G$. 
	By Lemma \ref{regular suborder}, $G_{\tau+1} = G \cap \p_{\tau+1}$ 
	is a generic filter on $\p_{\tau+1}$. 
	Note that 
	$\Ical_\tau = \bigcup \{ I_{p,\tau} : p \in G, \ \tau \in \dom(s_p) \}$ is 
	equal to $\bigcup \{ I_{u,\tau} : u \in G_{\tau+1}, \ \tau \in \dom(s_u) \}$. 
	By Lemmas \ref{K forces (B)} and \ref{two step equivalence}, 
	it is routine to check that $(B)_{\vec \pi,f_\tau}$ holds in $V[G]$ 
	as witnessed by $\Ical_\tau$.
\end{proof}

A similar argument combined with Lemma \ref{projection} shows the following.

\begin{lemma} \label{dense on coordinates}
	Let $\zeta < \delta \le \Delta$. 
	Suppose that $\dot D$ is a $\p_\zeta$-name for a dense subset of 
	$\dot \q_\zeta$. 
	Then for all $p \in \p_\delta$, there exists $q \le_\delta p$ such that 
	$\zeta \in \dom(s_q)$ and $q \res \zeta \Vdash_\zeta s_q(\zeta) \in \dot D$. 
\end{lemma}

The next two lemmas follow by Lemmas \ref{h is total} and \ref{dense on coordinates}.

\begin{lemma} \label{dense for h}
	Let $\delta \le \Delta$. 
	Let $Y$ be a countable subset of $\omega_2$. 
	Then for any $p \in \p_\delta$, there exists $q \le_\delta p$ 
	such that for any $\tau$ in $\dom(s_p)$, 
	$Y \subseteq \dom(h_{q,\tau})$.
\end{lemma}

\begin{lemma} \label{dense for c}
	Let $\delta \le \Delta$. 
	Let $X$ be a countable set of pairs of the form $(\beta,\gamma)$, where 
	$\omega_1 \le \beta < \omega_2$ and $\gamma < \beta$. 
	Then for any $p \in \p_\delta$, there exists $q \le_\delta p$ such that 
	$X \subseteq \dom(c_q)$.
\end{lemma}

\section{Preserving Cardinals, 1} \label{Preserving Cardinals, 1}

In order to prove that $\Delta = \omega_3$ and the recursion succeeds, 
we need to show that each $\p_\delta$ preserves all cardinals. 
By Lemma \ref{iteration is omega_1 closed}, 
we immediately have the preservation of $\omega_1$. 
Preserving cardinals greater than or equal to $\omega_3$ is handled next. 
The preservation of $\omega_2$ is the most difficult challenge of the article and 
is proven in Section \ref{Preserving Cardinals, 2}.

\begin{lemma} \label{centered}
	Assume that $\delta \le \Delta$ and $\delta < \omega_3$. 
	Then $\p_\delta$ is $\omega_2$-centered.
\end{lemma}

\begin{proof}
	We associate to any condition $p \in \p_\delta$ the following parameters:
	\begin{enumerate}
		\item $c_p$;
		\item $\dom(s_p)$;
		\item $\langle (h_{p,\tau},I_{p,\tau}) : \tau \in \dom(s_p) \rangle$.
	\end{enumerate}
	By \textsf{CH}, there are $\omega_2$-many possibilities for these three objects.

	Assume that $p_0,\ldots,p_{n-1}$ have the same such parameters. 
	Define $q$ as follows:
	\begin{itemize}
		\item $c_q = c_{p_0}$, $\dom(s_q) = \dom(s_{p_0})$, 
		and $A_q = \bigcup_{i < n} A_{p_i}$;
		\item for all $\tau \in \dom(s_q)$, 
		$s_q(\tau) = (h_{p_0,\tau},I_{p_0,\tau})$.
	\end{itemize}
	Using Lemma \ref{non-inductive characterization}, 
	it is simple to check that $q$ is in $\p_\delta$ and extends each of $p_0,\ldots,p_{n-1}$.
\end{proof}

\begin{prop} \label{chain condition}
	For all $\delta \le \Delta$, $\p_\delta$ is $\omega_3$-Knaster.
\end{prop}

\begin{proof}
	If $\delta < \omega_3$, then the statement is immediate by Lemma \ref{centered}. 
	Assume that $\delta = \Delta = \omega_3$. 
	Consider a sequence 
	$\langle p_i : i < \omega_3 \rangle$ of conditions in $\p_{\omega_3}$. 
	For each $i < \omega_3$, define 
	$\Mcal_i = \bigcup \{ M : \exists \xi \ (\xi,M) \in A_{p_i} \}$.

	Using $2^\omega = \omega_1$ and $2^{\omega_1} = \omega_2$, we can 
	in four successive steps:
	\begin{itemize}
		\item fix a stationary set 
		$A_0 \subseteq \omega_3$ consisting of ordinals of cofinality $\omega_2$ 
		and fix $\Mcal$ such that for all $i < j$ in $A_0$:
		\begin{itemize}
			\item $\Mcal_i \subseteq j$;
			\item $\Mcal_i \cap i = \Mcal$ and $\Mcal_j \cap j = \Mcal$.
		\end{itemize}
		\item fix a stationary set $A_1 \subseteq A_0$ 
		and fix $d$ such that for all $i < j$ in $A_1$:
		\begin{itemize}
			\item $\dom(s_{p_i}) \subseteq j$;
			\item $\dom(s_{p_i}) \cap i = \dom(s_{p_j}) \cap j = d$;
		\end{itemize}
		\item fix a stationary set $A_2 \subseteq A_1$ and 
		fix $c \in \C$ such that for all $i \in A_2$, $c_{p_i} = c$;
		\item fix a stationary set $A_3 \subseteq A_2$ and fix a sequence 
		$\langle (h_\tau,I_\tau) : \tau \in d \rangle$ 
		such that for all $i \in A_3$ and for all $\tau \in d$, 
		$h_{p_i,\tau} = h_\tau$ and $I_{p_i,\tau} = I_\tau$.
	\end{itemize}

	We claim that for all $i < j$ in $A_3$, $p_i$ and $p_j$ are compatible. 
	So consider $i < j$ in $A_3$. 
	Define $q = (c_q,s_q,A_q)$ as follows. 
	Let $c_q = c$, $A_q = A_{p_i} \cup A_{p_j}$, 
	and $\dom(s_q) = \dom(s_{p_i}) \cup \dom(s_{p_j})$. 
	For all $\tau \in \dom(s_q)$:
	\begin{itemize}
		\item if $\tau \in \dom(s_{p_i}) \setminus \dom(s_{p_j})$, let 
		$s_q(\tau) = s_{p_i}(\tau)$;
		\item if $\tau \in \dom(s_{p_j}) \setminus \dom(s_{p_i})$, let 
		$s_q(\tau) = s_{p_j}(\tau)$;
		\item if $\tau \in \dom(s_{p_i}) \cap \dom(s_{p_j})$, 
		let $s_q(\tau) = (h_\tau,I_\tau)$.
	\end{itemize}

	We show that $q$ is in $\p_\delta$ and extends $p_i$ and $p_j$. 
	Referring to Lemma \ref{non-inductive characterization}, the only non-trivial 
	thing to check is that for all $(\xi,M) \in A_q$ and for all 
	$\tau \in M \cap \dom(s_q) \cap \xi$, 
	$M \cap \omega_2 \in I_{q,\tau}$ (where $s_q(\tau) = (h_{q,\tau},I_{q,\tau})$). 
	If $(\xi,M) \in A_{p_i}$ and $\tau \in \dom(s_{p_i})$, or if 
	$(\xi,M) \in A_{p_j}$ and $\tau \in \dom(s_{p_j})$, 
	then we are done since $p_i$ and $p_j$ are conditions. 
	Assume that 
	$(\xi,M) \in A_{p_j}$ and $\tau \in \dom(s_{p_i}) \setminus \dom(s_{p_j})$. 
	Then $\tau < j$ so $\tau \in M \cap j \subseteq \Mcal \subseteq i$. 
	So $\tau \in \dom(s_{p_i}) \cap i = d \subseteq \dom(s_{p_j})$, which is a contradiction. 
	Now assume that $(\xi,M) \in A_{p_i}$ and $\tau \in \dom(s_{p_j}) \setminus \dom(s_{p_i})$. 
	Then $\tau \in M \subseteq \Mcal_i \subseteq j$, 
	so $\tau \in \dom(s_{p_j}) \cap j = d \subseteq \dom(s_{p_i})$, 
	which again is a contradiction.
\end{proof}

\begin{lemma} \label{preservation below is enough}
	Suppose that for all $\delta \le \Delta$ with $\delta < \omega_3$, 
	$\p_\delta$ preserves $\omega_2$. 
	Then $\Delta = \omega_3$ and $\p_{\omega_3}$ preserves $\omega_2$.
\end{lemma}

\begin{proof}
	The assumption clearly implies that $\Delta = \omega_3$, so it suffices 
	to show that $\p_{\omega_3}$ preserves $\omega_2$. 
	Let $\dot H$ be a nice $\p_{\omega_3}$-name for a function from $\omega_1$ into $\omega_2$. 
	Now every antichain of $\p_{\omega_3}$ has size at most $\omega_2$ by 
	Proposition \ref{chain condition}, and by definition every member of 
	$\p_{\omega_3}$ is a member of $\p_\delta$ for some $\delta < \omega_3$. 
	Consequently, every antichain of $\p_{\omega_3}$ is a subset of $\p_\delta$ 
	for some $\delta < \omega_3$. 
	Therefore, $\dot H$ is actually a nice $\p_{\beta}$-name for some $\beta < \omega_3$. 
	Since $\p_\beta$ preserves $\omega_2$, it forces that $\dot H$ is bounded below $\omega_2$. 
	As $\p_\beta$ is a regular suborder of $\p_{\omega_3}$, $\p_{\omega_3}$ 
	forces the same.
\end{proof}

\section{Density Lemmas} \label{Density Lemmas}

We now move towards the final goal of the article which is the preservation 
of $\omega_2$ by the forcing iteration. 
Before we jump into the proof, we first do some preparation by isolating 
several technical lemmas which we use in the next section.

For the remainder of the section, fix some $\delta \le \Delta$ with $\delta < \omega_3$.

\begin{lemma} \label{dense for range}
	For any $p \in \p_\delta$, there exists $q \le_\delta p$ such that 
	for all $(\beta,\gamma) \in \dom(c_p)$ and for all $\tau \in \dom(s_p)$, 
	if there exists some $w \le_\delta q$ such that $\gamma \in \ran(h_{w,\tau})$, 
	then $\gamma \in \ran(h_{q,\tau})$.
\end{lemma}

\begin{proof}
	Let $X$ be the set of all pairs $(\tau,\gamma)$ such that $\tau$ is 
	in $\dom(s_p)$ and there exists 
	some $\beta$ such that $(\beta,\gamma) \in \dom(c_p)$. 
	Injectively enumerate $X$ as $\langle (\tau_n,\gamma_n) : n < k \rangle$ 
	where $k \le \omega$. 
	Define by induction a sequence of conditions $\langle p_n : n \le k \rangle$ as follows. 
	Let $p_0 = p$. 
	Suppose that $n < k$ and $p_n$ is defined. 
	If there exists some $w \le_\delta p_n$ such that 
	$\gamma_n \in \ran(h_{w,\tau_n})$, 
	then let $p_{n+1}$ be any such condition, and otherwise let 
	$p_{n+1} = p_n$. 
	This completes the construction. 
	If $k < \omega$, let $q = p_k$, and if $k = \omega$, then let 
	$q$ be the greatest lower bound of the sequence $\langle p_n : n < \omega \rangle$ 
	using Lemma \ref{iteration is omega_1 closed}.
	It is straightforward to check that $q$ is as required.
\end{proof}

\begin{defn}
	For any $p \in \p_\delta$, define 
	$I_p = \bigcup \{ I_{p,\tau} : \tau \in \dom(s_p) \}$.
\end{defn}

\begin{defn}
	For any $p \in \p_\delta$ and for any $\beta \in I_p$, 
	define $c_{p,\beta}$ to be the function whose domain is equal to the set 
	$\{ \gamma : (\beta,\gamma) \in \dom(c_p) \}$ such that for any $\gamma$ 
	in the domain of $c_{p,\beta}$, $c_{p,\beta}(\gamma) = c_p(\beta,\gamma)$.
\end{defn}

\begin{defn} \label{definition of E_delta}
	Define $E_\delta$ to be the set of conditions $r \in \p_\delta$ satisfying that 
	for all $\beta \in I_r$ and for all $\tau \in \dom(s_r)$:
		\begin{enumerate}
			\item $\dom(c_{r,\beta}) \subseteq \dom(h_{r,\tau})$;
			\item $\dom(h_{r,\tau}) \cap \beta$ and 
			$\ran(h_{r,\tau}) \cap \beta$ are subsets of $\dom(c_{r,\beta})$;
			\item for all $\gamma \in \dom(c_{r,\beta})$, 
			if there exists some 
			$w \le_\delta r$ such that $\gamma \in \ran(h_{w,\tau})$, then 
			$\gamma \in \ran(h_{r,\tau})$.
		\end{enumerate}
\end{defn}

\begin{lemma} \label{E_delta is dense}
	The set $E_\delta$ is dense in $\p_\delta$.
\end{lemma}

\begin{proof}
	Let $p \in \p_\delta$, and we find $r \le_\delta p$ in $E_\delta$. 
	Define $\langle p_n : n < \omega \rangle$ by recursion as follows. 
	Let $p_0 = p$. 
	Let $n < \omega$ and assume that $p_n$ is defined. 
	Define $Y_n$ to be the set of all $\gamma$ such that for some $\beta \in I_{p_n}$, 
	$(\beta,\gamma) \in \dom(c_{p_n})$. 
	By Lemma \ref{dense for h}, fix $u_n \le_\delta p_n$ 
	such that for any $\tau \in \dom(s_{p_n})$, 
	$Y_n \subseteq \dom(h_{u_n,\tau})$. 
	Define $X_n$ to be the set of all pairs $(\beta,\gamma)$ such that 
	$\beta \in I_{u_n}$ and for some $\tau \in \dom(s_{u_n})$, 
	$\gamma \in (\dom(h_{u_n,\tau}) \cup \ran(h_{u_n,\tau})) \cap \beta$. 
	By Lemma \ref{dense for c}, fix $v_n \le_\delta u_n$ such that 
	$X_n \subseteq \dom(c_{v_n})$. 
	Now apply Lemma \ref{dense for range} to fix 
	$p_{n+1} \le_\delta v_n$ satisfying that for all 
	$\beta \in I_{v_n}$, for all $\gamma \in \dom(c_{v_n,\beta})$, 
	and for all $\tau \in \dom(s_{v_n})$, if there exists 
	some $w \le_\delta p_{n+1}$ such that $\gamma \in \ran(h_{w,\tau})$, 
	then $\gamma \in \ran(h_{p_{n+1},\tau})$. 
	Let $r$ be the greatest lower bound of $\langle p_n : n < \omega \rangle$. 
	It is routine to check that $r \in E_\delta$.
\end{proof}

\begin{lemma} \label{E_delta is dense 2}
	Suppose that $\omega \le \delta$. 
	Let $p \in E_\delta$ and suppose that $(\delta^{-},N) \in A_{p}$. 
	Then there exists $q \le_\delta p$ such that $q \in E_\delta$ and for all 
	$(\xi,M) \in A_q$, if $M \cap \omega_2 < N \cap \omega_2$ then 
	$(\sup(M \cap N \cap \xi),M \cap N) \in A_q$.
\end{lemma}

\begin{proof}
	Define $q = (c_q,s_q,A_q)$ by letting $c_q = c_p$, $s_q = s_p$, and 
	$$
	A_q = A_p \cup \{ (\sup(K \cap N \cap \sigma),K \cap N) : (\sigma,K) \in A_p, \ 
	K \cap \omega_2 < N \cap \omega_2 \}.
	$$
	We claim that $q$ is in $\p_\delta$ and extends $p$. 
	The only non-trivial thing to check is (5) of 
	Lemma \ref{non-inductive characterization}. 
	Consider $(\sigma,K) \in A_p$ with 
	$K \cap \omega_2 < N \cap \omega_2$ and 
	$\tau \in (K \cap N) \cap \dom(s_q) \cap \sup(K \cap N \cap \sigma)$. 
	Then $\tau \in K \cap \dom(s_p) \cap \sigma$, so 
	$K \cap \omega_2 = K \cap N \cap \omega_2$ is in $I_{p,\tau} = I_{q,\tau}$. 
	As $c_q = c_p$, $s_q = s_p$, and $p \in E_\delta$, by the definition of 
	$E_\delta$ it easily 
	follows that $q \in E_\delta$.

	Suppose that $(\xi,M) \in A_q$ and $M \cap \omega_2 < N \cap \omega_2$. 
	We show that $(\sup(M \cap N \cap \xi),M \cap N) \in A_q$. 
	If $(\xi,M) \in A_p$, then we are done by the definition of $q$. 
	Otherwise, for some $(\sigma,K) \in A_p$ with $K \cap \omega_2 < N \cap \omega_2$, 
	$(\xi,M) = (\sup(K \cap N \cap \sigma),K \cap N)$. 
	Hence, $\xi = \sup(K \cap N \cap \sigma)$ and $M = K \cap N$. 
	So $\sup(M \cap N \cap \xi) = 
	\sup(K \cap N \cap \sup(K \cap N \cap \sigma))$, which is easily shown 
	to be equal to $\sup(K \cap N \cap \sigma) = \xi$. 
	And also $M \cap N = M$. 
	So $(\sup(M \cap N \cap \xi),M \cap N) = (\xi,M) \in A_q$.
\end{proof}

\begin{lemma} \label{final dense set}
	Let $p \in \p_\delta$ and let $x$ be a subset of $\dom(s_p)$. 
	Let $\zeta < \omega_1$. 
	Then there exist $q \le_\delta p$, 
	countable limit ordinals $\zeta_0$ and $\zeta_1$, and a sequence 
	$\langle F_\tau : \tau \in x \rangle$ satisfying:
	\begin{enumerate}
		\item $\zeta \le \zeta_0 < \zeta_0 + \omega \le \zeta_1$;
		\item for all $\tau \in x$:
		\begin{enumerate}
			\item $F_\tau : \zeta_1 \to \zeta_1$;
			\item $q \res \tau \Vdash_\tau \dot f_\tau \res \zeta_1 = \check F_\tau$;
			\item $F_\tau[\zeta] \subseteq \zeta_0$;
		\end{enumerate}
		\item for all $j < \zeta_1$, the set 
		$\{ F_\tau(j) : \tau \in x \}$ is bounded below $\zeta_1$.
	\end{enumerate}	
\end{lemma}

\begin{proof}
	Note that by Lemmas \ref{dense set} and \ref{projection}, 
	for any $i < \omega_1$ and for any $\tau < \delta$, 
	the set of $q \in \p_\delta$ such that $q \res \tau$ decides (in $\p_\tau$) 
	$\dot f_\tau \res i$ is dense in $\p_\delta$. 
	Define by recursion sequences $\langle p_n : n < \omega \rangle$, 
	$\langle \iota_n : n < \omega \rangle$, and 
	$\langle F_{\tau,n} : \tau \in x, \ n < \omega \rangle$ as follows. 
	Using the fact stated above 
	together with Lemma \ref{iteration is omega_1 closed}, we can 
	find $p_0 \le_\delta p$ such that for all $\tau \in x$, 
	there exists $F_{\tau,0}$ such that 
	$p_0 \res \tau \Vdash_\tau \dot f_\tau \res \zeta = F_{\tau,0}$. 
	Let $\iota_0 = \zeta$.

	Let $n < \omega$ and assume that we have defined 
	$p_n$, $\iota_n$, and $F_{\tau,n}$ for all $\tau \in x$. 
	Also, assume as an inductive hypothesis that $\iota_n < \omega_1$ 
	and for all $\tau \in x$, 
	$p_n \res \tau \Vdash_\tau \dot f_\tau \res \iota_n = F_{\tau,n}$. 
	Choose a countable ordinal $\iota_{n+1} > \iota_n$ large enough so that 
	for all $\tau \in x$, $F_{\tau,n}[\iota_n] \subseteq \iota_{n+1}$. 
	Now apply the fact described in the previous paragraph to 
	fix some $p_{n+1} \le_\delta p_n$ and a sequence 
	$\langle F_{\tau,n+1} : \tau \in x \rangle$ 
	such that for all $\tau \in x$, 
	$p_{n+1} \res \tau \Vdash_\tau \dot f_\tau \res \iota_{n+1} = F_{\tau,n+1}$. 
	
	This completes the construction. 
	Let $q$ be the greatest lower bound of $\langle p_n : n < \omega \rangle$ 
	and for each $\tau \in x$, let $F_{\tau} = \bigcup_n F_{\tau,n}$. 
	Define $\zeta_0 = \iota_1$ and $\zeta_1 = \sup_n \iota_n$. 
	It is straightforward to check that these objects satisfy the required properties.
\end{proof}

\section{Preserving Cardinals, 2} \label{Preserving Cardinals, 2}

We now prove the preservation of $\omega_2$ by the forcing iteration, 
which completes our work. 

\begin{defn}
	Let $1 \le \delta \le \Delta$ with $\delta < \omega_3$. 
	Suppose that $N \in \Ycal^+$, $p \in N \cap \p_\delta$, and 
	$[\delta^{-},\delta) \subseteq \dom(s_p)$. 
	Define $p + (\delta,N)$ to be the ordered triple $(c,s,A)$ satisfying:
	\begin{enumerate}
		\item $c = c_p$;
		\item $\dom(s) = \dom(s_p)$;
		\item for all $\tau \in \dom(s)$:
		\begin{enumerate}
			\item if $\tau \notin N$, then $s(\tau) = s_p(\tau)$;
			\item if $\tau \in N$, then 
			$s(\tau) = (h_{p,\tau},I_{p,\tau} \cup \{ N \cap \omega_2 \})$;
		\end{enumerate}
		\item if $\delta < \omega$ then $A = A_p$, and if $\delta \ge \omega$ then 
		$A = A_p \cup \{ (\delta^{-},N) \}$.
	\end{enumerate}
\end{defn}

\begin{lemma}
	Let $1 \le \delta \le \Delta$ with $\delta < \omega_3$ and 
	let $N \in \Ycal^+$ with $\delta \in N$. 
	For any $p \in N \cap \p_\delta$ such that 
	$[\delta^{-},\delta) \subseteq \dom(s_p)$, 
	$p + (\delta,N)$ is a condition in $\p_\delta$ which extends $p$.
\end{lemma}

The proof is routine using Lemma \ref{non-inductive characterization}.

\begin{lemma} \label{point of p + (delta,N)}
	Let $1 \le \delta \le \Delta$ with $\delta < \omega_3$. 
	Suppose that $N \in \Ycal^+$ with $\delta \in N$, 
	$p \in N \cap \p_\delta$, and 
	$[\delta^{-},\delta) \subseteq \dom(s_p)$. 
	If $r \le_\delta p + (\delta,N)$, then for all $\tau \in N \cap \dom(s_r)$, 
	$N \cap \omega_2 \in I_{r,\tau}$.
\end{lemma}

\begin{proof}
	Let $\tau \in N \cap \dom(s_r)$. 
	If $\tau \in \dom(s_p)$, then by the definition of $p + (\delta,N)$, 
	$N \cap \omega_2 \in I_{p + (\delta,N),\tau} \subseteq I_{r,\tau}$. 
	Suppose that $\tau \in \dom(s_r) \setminus \dom(s_p)$. 
	If $\delta \ge \omega$ and $\tau < \delta^{-}$, then since $(\delta^{-},N) \in A_r$, 
	$N \cap \omega_2 \in I_{r,\tau}$ since $r$ is a condition. 
	Otherwise, either $\delta^{-} = 0 \le \tau < \delta < \omega$, 
	or else $\delta \ge \omega$ and 
	$\delta^{-} \le \tau < \delta$. 
	But in either case, $\tau \in \dom(s_p)$, which is a contradiction.
\end{proof}

The next theorem is essentially the heart of the article.

\begin{thm} \label{generic condition}
	Let $1 \le \delta \le \Delta$ with $\delta < \omega_3$ and 
	let $\chi > \omega_3$ be a regular cardinal. 
	Suppose that $N^* \prec H(\chi)$ satisfies that 
	$N = N^* \cap H(\omega_3) \in \Ycal^+$ and 
	$\langle \p_\beta : \beta \le \delta \rangle \in N^*$. 
	Assume that $p \in N \cap \p_\delta$ and 
	$[\delta^{-},\delta) \subseteq \dom(s_p)$. 
	Then $p + (\delta,N)$ is $(N^*,\p_\delta)$-generic.
\end{thm}

The proof of this result actually shows that $p + (\delta,N)$ 
is strongly $(N^*,\p_\delta)$-generic, 
but we do not need this fact for our purposes.

Before embarking on the difficult proof of Theorem \ref{generic condition}, 
we first observe that by standard proper forcing arguments this theorem allows us to 
complete the proof of the main result of the article.

\begin{prop} \label{preserves omega_2}
	For any $\delta \le \Delta$, $\p_\delta$ preserves $\omega_2$.
\end{prop}

\begin{proof}
	By Lemma \ref{preservation below is enough}, we may assume that $\delta < \omega_3$. 
	If $\delta = 0$, then $\p_\delta$ is isomorphic to $\C$ which is $\omega_2$-c.c. 
	So assume that $1 \le \delta < \omega_3$. 
	Let $\dot F$ be a $\p_\delta$-name for a function from $\omega_1$ into $\omega_2$. 
	We show that for all $p \in \p_\delta$, there exists $q \le_\delta p$ 
	which forces that $\dot F$ is bounded below $\omega_2$. 
	Let $p \in \p_\delta$. 
	By extending $p$ further if necessary using Lemma \ref{domain dense}, we may assume that 
	$[\delta^{-},\delta) \subseteq \dom(s_p)$. 
	Fix a regular cardinal $\chi > \omega_3$ such that $\dot F \in H(\chi)$. 
	As $\Ycal^+$ is a stationary subset of $[H(\omega_3)]^{\omega_1}$, we can fix 
	$N^* \prec H(\chi)$ such that 
	$\langle \p_\beta : \beta \le \delta \rangle$, $\dot F$, and $p$ are members of $N^*$ 
	and $N = N^* \cap H(\omega_3) \in \Ycal^+$.

	Since $\delta < \omega_3$, $p \in N^* \cap H(\omega_3) = N$. 
	So $p + (\delta,N)$ is defined, is a member of $\p_\delta$, and extends $p$. 
	By Theorem \ref{generic condition}, 
	$p + (\delta,N)$ is $(N^*,\p_\delta)$-generic. 
	We claim that $p + (\delta,N)$ forces that $\dot F$ is bounded below $\omega_2$.  
	Consider a generic filter $G$ on $\p_\delta$ which contains $p + (\delta,N)$. 
	Let $F = \dot F^G$, which is a member of $N^*[G]$. 
	As $\p_\delta \in N^* \prec H(\chi)$, $N^*[G]$ is an elementary substructure of 
	$H(\chi)^{V[G]}$, and as $\dot F \in N^*$, $F \in N^*[G]$ 
	and hence $N^*[G]$ is closed under $F$. 
	Since $\omega_1 \subseteq N^*[G]$, $F[\omega_1] \subseteq N^*[G] \cap \omega_2$. 
	As $p + (\delta,N)$ is in $G$ and is $(N^*,\p_\delta)$-generic, it follows that 
	$N^*[G] \cap \omega_2 = N \cap \omega_2 < \omega_2$, 
	so $F$ is bounded by $N \cap \omega_2$.
\end{proof}

\begin{cor}
	$\Delta = \omega_3$ and $\p_{\omega_3}$ preserves all cardinals.
\end{cor}

\begin{proof}
	By Lemmas \ref{iteration is omega_1 closed} and 
	\ref{preservation below is enough} together with Propositions 
	\ref{chain condition} and \ref{preserves omega_2}.
\end{proof}

\begin{prop} \label{I is stationary}
	For all $\tau < \omega_3$, $\p_{\omega_3}$ forces that 
	the set $\dot{\Ical}_\tau = 
	\bigcup \{ I_{p,\tau} : p \in \dot G_{\omega_3}, \ \tau \in \dom(s_p) \}$ 
	is stationary in $\omega_2$, where $\dot G_{\omega_3}$ is the $\p_{\omega_3}$-name 
	for the generic filter.
\end{prop}

\begin{proof}
	Let $\dot C$ be a nice $\p_{\omega_3}$-name for a club subset of $\omega_2$. 
	Fix $\delta < \omega_3$ large enough so that $\dot C$ is a $\p_\delta$-name. 	
	We show that for all $p \in \p_{\delta}$, there exists $q \le_{\delta} p$ 
	which forces that $\dot{\Ical}_\tau \cap \dot C$ is non-empty. 
	Let $p \in \p_\delta$. 
	By extending $p$ further if necessary using Lemma \ref{domain dense}, we may assume that 
	$[{\delta}^{-},\delta) \subseteq \dom(s_p)$ and also $\tau \in \dom(s_p)$. 
	Fix a regular cardinal $\chi > \omega_3$ such that $\dot C \in H(\chi)$. 
	Since $\Ycal^+$ is a stationary subset of $[H(\omega_3)]^{\omega_1}$, we can fix 
	$N^* \prec H(\chi)$ such that  
	$\langle \p_\beta : \beta \le \delta \rangle$, $\dot C$, $\tau$, and $p$ are members of $N^*$  
	and $N = N^* \cap H(\omega_3) \in \Ycal^+$. 
	Let $q = p + (\delta,N)$. 
	Then $q \in \p_\delta$, $q \le_\delta p$, and by Theorem \ref{generic condition}, 
	$q$ is $(N^*,\p_\delta)$-generic. 
	We have that $N \cap \omega_2 \in I_{q,\tau}$, so 
	$q$ forces that $N \cap \omega_2 \in \dot{\Ical}_\tau$. 
	Let $G$ be a generic filter on $\p_\delta$ containing $q$ 
	and let $C = \dot C^G$. 
	Then $C \in N^*[G]$. 
	Since $N^*[G]$ is an elementary substructure of $H(\chi)^{V[G]}$ and 
	$q$ is $(N^*,\p_\delta)$-generic, 
	$N^* \cap \omega_2 = N \cap \omega_2$ is a limit point of $C$ and hence is 
	a member of $C$.
\end{proof}

\begin{cor}
	For all $\delta < \omega_3$, $\p_{\omega_3}$ forces that 
	$(B)_{\dot{\vec \pi},\dot f_\delta}$ holds.
\end{cor}

\begin{proof}
	By Propositions \ref{forcing (B)} and \ref{I is stationary}.
\end{proof}

\begin{proof}[Proof of Theorem \ref{generic condition}]
	Let $\theta = N \cap \omega_2$. 
	Then $\theta \in S^2_1$. 
	Fix a dense open subset $D$ of $\p_\delta$ which is a member of $N^*$, 
	and we prove that $N^* \cap D$ is predense below $p + (\delta,N)$. 
	Let $r \le_\delta p + (\delta,N)$ be given, 
	and we find a member of $N^* \cap D$ which is compatible with $r$. 
	By extending $r$ further if necessary using Lemmas 
	\ref{E_delta is dense} and \ref{E_delta is dense 2}, 
	we may assume without loss of generality 
	that $r \in D \cap E_\delta$ and 
	for all $(\xi,M) \in A_r$, if $M \cap \omega_2 < N \cap \omega_2$ then 
	$(\sup(M \cap N \cap \xi),M \cap N) \in A_r$. 
	Define $Z = \{ \gamma : \exists \beta \in I_{r} \ (\beta,\gamma) \in \dom(c_r) \}$.

	Fix $\theta_0 < \theta$ and $\zeta < \omega_1$ 
	both large enough so that:
	\begin{enumerate}
		\item[(1)] $I_{r} \cap \theta \subseteq \theta_0$;
		\item[(2)] $Z \cap \theta \subseteq \theta_0$;
		\item[(3)] for all $\tau \in \dom(s_r)$, 
		$\dom(h_{r,\tau}) \cap \theta$ and $\ran(h_{r,\tau}) \cap \theta$ 
		are subsets of $\theta_0$;
		\item[(4)] $\ran(c_r) \subseteq \zeta$.
	\end{enumerate}
	Since $r \le_\delta p + (\delta,N)$, by Lemma \ref{point of p + (delta,N)} it follows that 
	for all $\tau \in N \cap \dom(s_r)$, $\theta \in I_{r,\tau}$, and hence 
	$\theta$ is closed under both $h_{r,\tau}$ and $h_{r,\tau}^{-1}$.

	Define:
	\begin{enumerate}
		\item[(5)] $c_N = c_r \res N$;
		\item[(6)] $d_N = \dom(s_r) \cap N$;
		\item[(7)] $A_N = A_r \cap N$;
		\item[(8)] $I_N = I_{r} \cap N$;
		\item[(9)] $Z_N = Z \cap N$;
		\item[(10)] for all $\tau \in d_N$, $h_{N,\tau} = h_{r,\tau} \res N$ and 
		$I_{N,\tau} = I_{r,\tau} \cap N$.
	\end{enumerate}

	Applying the elementarity of $N^*$, fix $\bar r$ and 
	$\bar \theta$ which are members of $N^*$ and satisfy:
	\begin{enumerate}
		\item[(11)] $\bar r \in D \cap E_\delta$;
		\item[(12)] $d_N \subseteq \dom(s_{\bar r})$ and 
		$A_N \subseteq A_{\bar r}$;
		\item[(13)] for all $\gamma \in Z_N$, there exists $\beta \in I_{\bar r}$ 
		such that $(\beta,\gamma) \in \dom(c_{\bar r})$;
		\item[(14)] $\bar \theta \in S^2_1$, $\theta_0 < \bar \theta$, 
		and for all $\tau \in d_N$, 
		$\bar \theta \in I_{\bar r,\tau}$;
		\item[(15)] $c_{\bar r} \res \bar{\theta}^2 = c_N$;
		\item[(16)] for all $\tau \in d_N$:
		\begin{enumerate}
			\item $\bar \theta$ is closed under 
			$h_{\bar r,\tau}$ and $h_{\bar r,\tau}^{-1}$;
			\item $h_{\bar r,\tau} \res \bar{\theta} = h_{N,\tau}$;
			\item $I_{\bar r,\tau} \cap \bar{\theta} = I_{N,\tau}$.
		\end{enumerate}
	\end{enumerate}
	Namely, the properties described in (11)--(16) can be expressed by a first-order 
	formula with parameters in $N^*$ (using the fact that $N^\omega \subseteq N$) 
	which is satisfied by $r$ and $\theta$ in $H(\chi)$. 
	So by elementarity, these properties are satisfied by some 
	$\bar r$ and $\bar \theta$ in $N^*$. 
	Note that $\bar r$ and $\bar \theta$ are in $H(\omega_3)$ and hence are in $N$.

	Applying Lemma \ref{final dense set}, 
	fix in $N^*$ a condition $v \le_\delta \bar r$, countable limit ordinals 
	$\zeta_0$ and $\zeta_1$, and a sequence 
	$\langle F_\tau : \tau \in d_N \rangle$ satisfying:
	\begin{enumerate}
		\item[(17)] $\zeta \le \zeta_0 < \zeta_0 + \omega \le \zeta_1$;
		\item[(18)] for all $\tau \in d_N$:
		\begin{enumerate}
			\item $F_\tau : \zeta_1 \to \zeta_1$;
			\item $v \res \tau \Vdash_\tau \dot f_\tau \res \zeta_1 = \check F_\tau$;
			\item $F_\tau[\zeta] \subseteq \zeta_0$;
		\end{enumerate}
		\item[(19)] for all $j < \zeta_1$, the set 
		$\{ F_\tau(j) : \tau \in d_N \}$ is bounded below $\zeta_1$.
	\end{enumerate}

	Since $D$ is open, $v \in D \cap N^*$. 
	So we are done provided we can show that $v$ and $r$ are compatible. 

	\bigskip

	\textbf{Claim A:} Suppose that $\beta \in I_{r} \setminus \theta$, 
	$\tau \in d_N$, and $\alpha \in \dom(h_{v,\tau})$. 
	If $h_{v,\tau}(\alpha) \in \dom(c_{r,\beta})$, then 
	$\alpha \in \dom(c_{r,\beta})$.

	\emph{Proof:} 
	Assume that $h_{v,\tau}(\alpha) \in \dom(c_{r,\beta})$. 
	Then $h_{v,\tau}(\alpha) \in Z_N$. 
	So by (13), there exists $\bar \beta \in I_{\bar r}$ such that 
	$h_{v,\tau}(\alpha) \in \dom(c_{\bar r,\bar \beta})$. 
	Since $v \le_\delta \bar r$ and $h_{v,\tau}(\alpha) \in \ran(h_{v,\tau})$, 
	the fact that $\bar r \in E_\delta$ implies 
	by Definition \ref{definition of E_delta}(3) 
	that $h_{v,\tau}(\alpha) \in \ran(h_{\bar r,\tau})$. 
	Since $h_{\bar r,\tau} \subseteq h_{v,\tau}$ and $h_{v,\tau}$ is injective, 
	it follows that $\alpha \in \dom(h_{\bar r,\tau})$ and 
	$h_{\bar r,\tau}(\alpha) = h_{v,\tau}(\alpha)$. 
	Now $h_{v,\tau}(\alpha) \in Z_N \subseteq \theta_0 < \bar{\theta}$. 
	Since $\bar{\theta} \in I_{\bar r,\tau} \subseteq I_{v,\tau}$, 
	$\alpha < \bar{\theta}$ since $v$ being a condition implies that 
	$\bar{\theta}$ is closed under $h_{v,\tau}^{-1}$. 
	But $h_{\bar r,\tau} \res \bar{\theta} = h_{N,\tau} = h_{r,\tau} \res \bar{\theta}$. 
	So $\alpha \in \dom(h_{r,\tau})$ and $\alpha < \beta$. 
	Since $r \in E_\delta$, by Definition \ref{definition of E_delta}(2) 
	we have that $\alpha \in \dom(c_{r,\beta})$. 
	This completes the proof of Claim A.

	\bigskip

	\textbf{Definitions of $C$ and functions 
	$C_\beta : \{ \gamma_{\beta,n} : n < k_\beta \} \to \zeta_1$ 
	for each $\beta \in I_r \setminus \theta$.} 
	For each $\beta \in I_r \setminus \theta$, we define a function 
	$C_\beta$ with domain equal to the set 
	$$
	\left( \bigcup \{ \dom(h_{v,\tau}) \cup \ran(h_{v,\tau}) : 
	\tau \in d_N \} \right) 
	\setminus \dom(c_{r,\beta})
	$$
	and mapping into the interval $[\zeta_0,\zeta_1)$ as follows. 
	Injectively enumerate the domain of $C_\beta$ as 
	$\langle \gamma_{\beta,n} : n < k_\beta \rangle$, 
	where $k_\beta \le \omega$. 
	We define $C_\beta(\gamma_{\beta,n})$ by induction on $n < k_\beta$. 
	Suppose that $n < k_\beta$ and for all $0 \le m < n$, $C_\beta(\gamma_{\beta,m})$ is defined. 
	Define $C_\beta(\gamma_{\beta,n})$ to be any ordinal in the interval 
	$[\zeta_0,\zeta_1)$ which is greater than any member of the set 
	$$
	\{ C_\beta(\gamma_{\beta,m}) : m < n \} \cup 
	\{ F_\tau(C_\beta(\gamma_{\beta,m})) : m < n, \ \tau \in d_N \}.
	$$
	This is possible by statement (19) above. 
	This completes the definition of $C_\beta$. 
	Now define $C$ with domain equal to the set of all $(\beta,\gamma)$ with 
	$\beta \in I_r \setminus \theta$ and $\gamma \in \dom(C_\beta)$ so that 
	for all such $(\beta,\gamma)$, 
	$C(\beta,\gamma) = C_\beta(\gamma)$.

	\bigskip

	\textbf{Claim B:} For all $\beta \in I_r \setminus \theta$, 
	for all $\gamma \in \dom(C_\beta)$, 
	and for all $\tau \in d_N$, if $\gamma \in \dom(h_{v,\tau})$ then 
	$h_{v,\tau}(\gamma) \in \dom(C_\beta)$.
	
	\emph{Proof:} 
	Immediate by Claim A.
		
	\bigskip
	 
	\textbf{Claim C:} For all $\beta \in I_r \setminus \theta$ and for all 
	$\tau \in d_N$:
	\begin{enumerate}
		\item if $\gamma \in \dom(h_{v,\tau}) \cap \dom(C_\beta)$, then 
		$$
		F_\tau(\min \{ C_\beta(\gamma), C_\beta(h_{v,\tau}(\gamma)) \}) < 
		\max \{ C_\beta(\gamma), C_\beta(h_{v,\tau}(\gamma) \};
		$$ 
		\item if $\alpha, \gamma$ are distinct elements of 
		$\dom(C_\beta) \cap \dom(h_{v,\tau})$, then 
		$$
		F_\tau(\min \{ C_\beta(h_{v,\tau}(\alpha)), 
		C_\beta(h_{v,\tau}(\gamma)) \}) < 
		\max \{ C_\beta(h_{v,\tau}(\alpha)), C_\beta(h_{v,\tau}(\gamma)) \}.
		$$
	\end{enumerate}

	\emph{Proof:} (1) Fix $n < k_\beta$ such that $\gamma = \gamma_{\beta,n}$. 
	By Claim B, $h_{v,\tau}(\gamma_{\beta,n}) \in \dom(C_\beta)$, 
	so we can fix $m < k_\beta$ different from $n$ such that 
	$h_{v,\tau}(\gamma_{\beta,n}) = \gamma_{\beta,m}$. 
	Case 1: $n < m$. 
	By the definition of $C_\beta$, 
	$C_\beta(\gamma_{\beta,n})$ and $F_\tau(C_\beta(\gamma_{\beta,n}))$ are both 
	less than $C_\beta(\gamma_{\beta,m}) = C_\beta(h_{v,\tau}(\gamma_{\beta,n}))$, 
	so (1) holds. 
	Case 2: $m < n$. 
	By the definition of $C_\beta$, 
	$C_\beta(\gamma_{\beta,m})$ and $F_\tau(C_\beta(\gamma_{\beta,m}))$ are both 
	less than $C_\beta(\gamma_{\beta,n})$.
	But this is the same as $C_\beta(h_{v,\tau}(\gamma_{\beta,n}))$ and 
	$F_\tau(C_\beta(h_{v,\tau}(\gamma_{\beta,n})))$ both being less 
	than $C_\beta(\gamma_{\beta,n})$, so (1) holds.
	 
	(2) Fix distinct $m, n < k_\beta$ such that 
	$\alpha = \gamma_{\beta,m}$ and $\gamma = \gamma_{\beta,n}$. 
	Since $h_{v,\tau}$ is injective, by Claim B we can fix 
	distinct $g, l < k_\beta$ such that 
	$h_{v,\tau}(\gamma_{\beta,m}) = \gamma_{\beta,g}$ and 
	$h_{v,\tau}(\gamma_{\beta,n}) = \gamma_{\beta,l}$. 
	By symmetry, we may assume that $g < l$. 
	Then by the definition of $C_\beta$, both 
	$C_\beta(\gamma_{\beta,g})$ and $F_\tau(C_\beta(\gamma_{\beta,g}))$ are 
	less than $C_\beta(\gamma_{\beta,l})$. 
	So $C_\beta(h_{v,\tau}(\gamma_{\beta,m}))$ and 
	$F_\tau(C_\beta(h_{v,\tau}(\gamma_{\beta,m})))$ are less than 
	$C_\beta(h_{v,\tau}(\gamma_{\beta,n}))$ and (2) holds.

	\bigskip

	We are now ready to define a condition $w$ which extends $v$ and $r$.

	\bigskip

	\textbf{Definition of $w$:} 
	Define $w = (c_w,s_w,A_w)$ as follows. 
	Let $c_w = c_v \cup c_r \cup C$ and $A_w = A_v \cup A_r$. 
	Let the domain of $s_w$ be equal to $\dom(s_v) \cup \dom(s_r)$. 
	For each $\tau \in \dom(s_w)$, define $s_w(\tau)$ as follows.
	\begin{itemize}
		\item[(I)] If $\tau \in \dom(s_v) \setminus \dom(s_r)$, then 
		let 
			$$
			s_w(\tau) = 
			(h_{v,\tau},I_{v,\tau} \cup \{ K \cap \omega_2 : \exists \beta \ 
			(\beta,K) \in A_r, \ \theta \le K \cap \omega_2, 
			\ \tau \in K \cap \beta \}).
			$$
		\item[(II)] If $\tau \in \dom(s_r) \setminus \dom(s_v)$, 
		let $s_w(\tau) = s_r(\tau)$.
		\item[(III)] If $\tau \in \dom(s_v) \cap \dom(s_r)$, let 
		$s_w(\tau) = (h_{v,\tau} \cup h_{r,\tau},I_{v,\tau} \cup I_{r,\tau})$.
	\end{itemize}

	\bigskip
	
	We split up the proof that $w$ is a condition which extends $v$ and $r$ into 
	a series of claims.
	
	\bigskip
	
	\textbf{Claim D:} $c_w \in \C$ and $c_v$ and $c_r$ are subsets of $c_w$.

	\emph{Proof:}  
	We start by showing that $c_w$ is a function, that is, whenever 
	$(\beta,\gamma)$ is in the domain of at least two of 
	$c_v$, $c_r$, or $C$, then these functions have the same value at $(\beta,\gamma)$. 
	The domain of $C$ is disjoint from the domains of $c_v$ and $c_r$. 
	Namely, if $(\beta,\gamma) \in \dom(C)$, then $\beta \notin N$ so $(\beta,\gamma)$ 
	is not in the domain of $c_v$, and also by the definition of the domain of $C_\beta$, 
	$\gamma \notin \dom(c_{r,\beta})$ so $(\beta,\gamma) \notin \dom(c_r)$. 
	We are left with the case that $(\beta,\gamma) \in \dom(c_v) \cap \dom(c_r)$. 
	But then $(\beta,\gamma) \in \dom(c_v) \cap N \subseteq \dom(c_N)$, 
	so $c_r(\beta,\gamma) = c_N(\beta,\gamma) = c_{\bar r}(\beta,\gamma) = c_v(\beta,\gamma)$.

	Now we show that for any distinct $(\beta,\alpha)$ and $(\beta,\gamma)$ 
	in the domain of $c_w$, $c_w(\beta,\alpha) \ne c_w(\beta,\gamma)$. 
	This is clear if $(\beta,\alpha)$ and $(\beta,\gamma)$ are both in the domain 
	of any of $c_v$, $c_r$, or $C$. 
	Case 1: $\beta \notin N$. 
	Then neither $(\beta,\alpha)$ nor $(\beta,\gamma)$ are in $\dom(c_v)$. 
	So without loss of generality we may assume that 
	$(\beta,\alpha) \in \dom(c_r)$ and $(\beta,\gamma) \in \dom(C)$. 
	Then $c_w(\beta,\alpha) = c_r(\beta,\alpha) < \zeta \le \zeta_0 \le 
	C(\beta,\gamma) = c_w(\beta,\gamma)$. 
	Case 2: $\beta \in N$. 
	Then neither $(\beta,\alpha)$ nor $(\beta,\gamma)$ are in $\dom(C)$. 
	Without loss of generality we may assume that 
	$(\beta,\alpha) \in \dom(c_r)$ and $(\beta,\gamma) \in \dom(c_v)$. 
	So $\beta \in I_r \cap N \subseteq \theta$. 
	Hence, $\alpha < \theta$ so $(\beta,\alpha) \in N$. 
	Therefore, $(\beta,\alpha) \in \dom(c_N) \subseteq \dom(c_{v})$ and we are done. 
	Finally, $c_v$ and $c_r$ are subsets of $c_w$ by definition. 
	This completes the proof of Claim D.
	
	\bigskip
	
	The next two claims are simple to verify.

	\bigskip
	
	\textbf{Claim E:} $s_w$ is a function whose domain is a countable subset of $\delta$, 
	and $\dom(s_v)$ and $\dom(s_r)$ are subsets of $\dom(s_w)$.

	\bigskip

	\textbf{Claim F:} $A_w$ is a countable set of pairs of the form 
	$(\xi,M)$, where $\xi \le \delta$, $\xi < \omega_3$, $\xi$ is a limit ordinal, 
	and $M \in \Ycal$. 
	Moreover, $A_v$ and $A_r$ are subsets of $A_w$.

	\bigskip
	
	\textbf{Claim G:} For all $\tau \in \dom(s_w)$, $s_w(\tau)$ is an ordered pair 
	$(h_{w,\tau},I_{w,\tau})$ satisfying: 
	\begin{enumerate}
		\item[(i)] $h_{w,\tau}$ is an injective function whose domain is a countable 
		subset of $\omega_2$, which maps into $\omega_2$, 
		and satisfies that for all $\alpha \in \dom(h_{w,\tau})$, 
		$h_{w,\tau}(\alpha) \ne \alpha$;
		\item[(ii)] $I_{w,\tau}$ is a countable subset of $S^2_1$;
		\item[(iii)] for all $\alpha \in \dom(h_{w,\tau})$ and for all $\beta \in I_{w,\tau}$, 
		if $\alpha < \beta$ then $h_{w,\tau}(\alpha) < \beta$;
		\item[(iv)] for all $\gamma \in \ran(h_{w,\tau})$ and for all $\beta \in I_{w,\tau}$, 
		if $\gamma < \beta$ then $h_{w,\tau}^{-1}(\gamma) < \beta$.
	\end{enumerate}
	
	\emph{Proof:} 
	(i) The fact that $h_{w,\tau}$ satisfies (i) is immediate in cases 
	(I) and (II) of the definition of $s_w(\tau)$. 
	For case (III), suppose that 
	$\tau \in \dom(s_v) \cap \dom(s_r)$ and so 
	$h_{w,\tau} = h_{v,\tau} \cup h_{r,\tau}$. 
	Recall that $h_{r,\tau} \res \theta = h_{N,\tau} \subseteq h_{v,\tau}$, 
	so $h_{w,\tau} = h_{v,\tau} \cup (h_{r,\tau} \res [\theta,\omega_2))$. 
	Therefore, $h_{w,\tau}$ is indeed a function. 
	Now the range of $h_{v,\tau}$ is a subset of $\theta$, whereas for all 
	$\alpha \in \dom(h_{r,\tau}) \setminus \theta$, 
	$h_{r,\tau}(\alpha) \ge \theta$. 
	Since $h_{v,\tau}$ and $h_{r,\tau}$ are injective, so is $h_{w,\tau}$. 
	Finally, $h_{w,\tau}(\alpha) \ne \alpha$ for all $\alpha \in \dom(h_{w,\tau})$ 
	is obvious.

	(ii) is immediate. 

	(iii) Suppose that $\alpha \in \dom(h_{w,\tau})$, $\beta \in I_{w,\tau}$, 
	and $\alpha < \beta$. 
	We show that $h_{w,\tau}(\alpha) < \beta$. 
	This is immediate in cases (I) and (II) of the definition of $s_w$. 
	So assume that $\tau \in \dom(s_v) \cap \dom(s_r)$. 
	We are done if either $\alpha \in \dom(h_{v,\tau})$ and $\beta \in I_{v,\tau}$, 
	or if $\alpha \in \dom(h_{r,\tau})$ and $\beta \in I_{r,\tau}$. 
	Case 1: $\alpha \in \dom(h_{v,\tau})$ and 
	$\beta \in I_{r,\tau}$. 
	Then $h_{w,\tau}(\alpha) = h_{v,\tau}(\alpha) < \theta$. 
	If $\beta \ge \theta$ we are done, so assume that $\beta < \theta$. 
	Then $\beta \in I_{r,\tau} \cap \theta = I_{N,\tau} \subseteq I_{v,\tau}$. 
	Case 2: $\alpha \in \dom(h_{r,\tau})$ and $\beta \in I_{v,\tau}$. 
	Then $\alpha < \beta < \theta$. 
	So $\alpha \in \dom(h_{r,\tau}) \cap N = \dom(h_{N,\tau}) \subseteq \dom(h_{v,\tau})$.

	(iv) Suppose that $\gamma \in \ran(h_{w,\tau})$, $\beta \in I_{w,\tau}$, 
	and $\gamma < \beta$. 
	We show that $h_{w,\tau}^{-1}(\gamma) < \beta$. 
	This is immediate in cases (I) and (II) of the definition of $s_w$. 
	So assume that $\tau \in \dom(s_v) \cap \dom(s_r)$. 
	We are done if either $\gamma \in \ran(h_{v,\tau})$ and $\beta \in I_{v,\tau}$, 
	or if $\gamma \in \ran(h_{r,\tau})$ and $\beta \in I_{r,\tau}$. 
	Case 1: $\gamma \in \ran(h_{v,\tau})$ and 
	$\beta \in I_{r,\tau}$. 
	If $\beta \ge \theta$ then we are done since $h_{v,\tau}^{-1}(\gamma) < \theta$. 
	Suppose that $\beta < \theta$. 
	Then $\beta \in I_{r,\tau} \cap N = I_{N,\tau} \subseteq I_{v,\tau}$. 
	Case 2: $\gamma \in \ran(h_{r,\tau})$ and $\beta \in I_{v,\tau}$. 
	Then $\gamma < \beta < \theta$. 
	Let $\alpha = h_{r,\tau}^{-1}(\gamma)$. 
	Then $\alpha < \theta$ since $\theta \in I_{r,\tau}$. 
	So $\alpha \in \dom(h_{r,\tau} \res \theta) = 
	\dom(h_{N,\tau}) \subseteq \dom(h_{v,\tau})$ and 
	$h_{v,\tau}(\alpha) = h_{r,\tau}(\alpha) = \gamma$. 
	So $\gamma \in \ran(h_{v,\tau})$.
	This completes the proof of Claim G.
	
	\bigskip
	
	\textbf{Claim H:} For all $(\xi,M) \in A_w$ and for all 
	$\tau \in M \cap \dom(s_w) \cap \xi$, 
	$M \cap \omega_2 \in I_{w,\tau}$.

	\emph{Proof:} If $(\xi,M) \in A_v$ and $\tau \in \dom(s_v)$, or if 
	$(\xi,M) \in A_r$ and $\tau \in \dom(s_r)$, 
	then we are done since $v$ and $r$ are conditions. 
	Case 1: $(\xi,M) \in A_v$ and $\tau \in \dom(s_r) \setminus \dom(s_v)$. 
	Then $M \in N$, so $M \subseteq N$. 
	Therefore, $\tau \in \dom(s_r) \cap N = d_N \subseteq \dom(s_v)$, 
	which is a contradiction. 
	Case 2: $(\xi,M) \in A_r$ and $\tau \in \dom(s_v) \setminus \dom(s_r)$. 
	So we are in case (I) in the definition of $s_w(\tau)$. 
	If $M \cap \omega_2 \ge \theta$, then $M \cap \omega_2 \in I_{w,\tau}$ 
	by the definition of $s_w(\tau)$. 
	Suppose that $M \cap \omega_2 < \theta$. 
	Then by the choice of $r$, $(\sup(M \cap N \cap \xi),M \cap N) \in A_r$. 
	Now $M \cap N \in N$ and $M \cap N \cap \xi$ is an initial segment of 
	$M \cap N \cap \delta$. 
	Since $M \cap N \cap \delta$ is in $N$ and has $\omega_1$-many initial segments, 
	$M \cap N \cap \xi$ is in $N$ and so $\sup(M \cap N \cap \xi) \in N$. 
	Hence, $(\sup(M \cap N \cap \xi),M \cap N) \in A_r \cap N = A_N \subseteq A_v$. 
	As $\tau \in (M \cap N) \cap \dom(s_v) \cap \sup(M \cap N \cap \xi)$, it follows 
	that $M \cap N \cap \omega_2 = M \cap \omega_2$ is in $I_{v,\tau} \subseteq I_{w,\tau}$.
	This completes the proof of Claim H.
	
	\bigskip
	
	Putting Claims (D)--(H) together, by Lemma \ref{non-inductive characterization} 
	we have that $w \in \p_\delta$.
	
	\bigskip
	
	\textbf{Claim I:} For all $\tau \in \dom(s_v)$, 
	$w \res \tau \Vdash_\tau s_w(\tau) \le_{\dot \q_\tau} s_v(\tau)$.
	
	\emph{Proof:} 
	Case 1: $\tau \in \dom(s_v) \setminus \dom(s_r)$. 
	Then by definition, $h_{w,\tau} = h_{v,\tau}$ and $I_{v,\tau} \subseteq I_{w,\tau}$. 
	So Definition \ref{definition of the main forcing}(a, b) are immediate, 
	and (c) is vacuous because 
	$\dom(h_{w,\tau}) \setminus \dom(h_{v,\tau})$ is empty. 
	Case 2: $\tau \in \dom(s_v) \cap \dom(s_r)$. 
	Then by definition, $h_{v,\tau} \subseteq h_{w,\tau}$ and 
	$I_{v,\tau} \subseteq I_{w,\tau}$. 
	So Definition \ref{definition of the main forcing}(a, b) are immediate, 
	and (c) is trivial because 
	if $\alpha \in \dom(h_{w,\tau}) \setminus \dom(h_{v,\tau})$ 
	and $\beta \in I_{v,\tau}$, 
	then $\beta < \theta \le \alpha$ since 
	$\dom(h_{r,\tau}) \cap \theta = \dom(h_{N,\tau}) \subseteq \dom(h_{v,\tau})$.
	This completes the proof of Claim I.
	
	\bigskip
	
	Claims (D), (E), (F), and (I) establish that $w \le_\delta v$.
	
	\bigskip

	\textbf{Claim J:} For all $\tau \in \dom(s_r)$, 
	$w \res \tau \Vdash_\tau s_w(\tau) \le_{\dot \q_\tau} s_r(\tau)$.

	\emph{Proof:} If $\tau \in \dom(s_r) \setminus \dom(s_v)$, 
	then $s_w(\tau) = s_r(\tau)$ so the claim is immediate. 
	So we may assume that $\tau \in \dom(s_r) \cap \dom(s_v)$, which implies 
	that $\tau \in \dom(s_r) \cap N = d_N$ and 
	we are in case (III) of the definition of $s_w(\tau)$. 
	By definition, $h_{r,\tau} \subseteq h_{w,\tau}$ and $I_{r,\tau} \subseteq I_{w,\tau}$, 
	so (a) and (b) of Definition \ref{definition of the main forcing} are immediate. 

	For (c), suppose that $\alpha \in \dom(h_{w,\tau}) \setminus \dom(h_{r,\tau})$, 
	$\beta \in I_{r,\tau}$, and $\alpha < \beta$. 
	We need to show that $w \res \tau$ forces (i) and (ii) of 
	Definition \ref{definition of the main forcing}(c). 
	By case (III) in the definition of $s_w(\tau)$, 
	we have that 
	$s_w(\tau) = (h_{v,\tau} \cup h_{r,\tau},I_{v,\tau} \cup I_{r,\tau})$ 
	and so $\alpha \in \dom(h_{v,\tau})$.

	Case 1: $\beta < \theta$. 
	Then $\beta \in I_{r,\tau} \cap \theta = I_{N,\tau} \subseteq I_{\bar r,\tau}$ 
	and $\alpha < \beta < \theta_0$. 
	If $\alpha \in \dom(h_{\bar r,\tau})$, then since 
	$h_{\bar r,\tau} \res \theta_0 = h_{r,\tau} \res \theta_0$, 
	$\alpha \in \dom(h_{r,\tau})$, which is false. 
	So $\alpha \in \dom(h_{v,\tau}) \setminus \dom(h_{\bar r,\tau})$. 
	As $v \le_\delta \bar r$, we have that:
	$$
	v \res \tau \Vdash_\tau 
	\dot f_\tau(\min \{ \dot{\pi}_\beta(\alpha), \dot{\pi}_\beta(h_{v,\tau}(\alpha)) \}) < 
	\max \{ \dot{\pi}_\beta(\alpha), \dot{\pi}_\beta(h_{v,\tau}(\alpha)) \}.
	$$
	But $w \le_\delta v$, so $w \res \tau$ forces the same. 
	Also, $h_{v,\tau}(\alpha) = h_{w,\tau}(\alpha)$. 
	So 
	$$
	w \res \tau \Vdash_\tau 
	\dot f_\tau(\min \{ \dot{\pi}_\beta(\alpha), \dot{\pi}_\beta(h_{w,\tau}(\alpha)) \}) < 
	\max \{ \dot{\pi}_\beta(\alpha), \dot{\pi}_\beta(h_{w,\tau}(\alpha)) \},
	$$
	which proves (i). 
	The same argument shows that for all $\gamma \in \dom(h_{v,\tau}) \cap \beta$ 
	different from $\alpha$, 
	$$
	w \res \tau \Vdash_\tau 
	\dot f_\tau(\min\{\dot{\pi}_\beta(h_{w,\tau}(\gamma)),\dot{\pi}_\beta(h_{w,\tau}(\alpha))\}) < 
	\max\{\dot{\pi}_\beta(h_{w,\tau}(\gamma)),\dot{\pi}_\beta(h_{w,\tau}(\alpha))\}.
	$$
	Assume that $\gamma \in \dom(h_{w,\tau}) \cap \beta$. 
	Then either $\gamma \in \dom(h_{v,\tau})$, or else 
	$\gamma \in \dom(h_{r,\tau}) \cap \beta \subseteq \dom(h_{v,\tau})$. 
	So in either case, $\gamma \in \dom(h_{v,\tau}) \cap \beta$ and the proof of 
	(ii) is complete.
	
	Case 2: $\beta \ge \theta$. 
	If $\alpha \in \dom(c_{r,\beta})$, then since $r \in E_\delta$ it follows 
	by Definition \ref{definition of E_delta}(1) that 
	$\alpha \in \dom(h_{r,\tau})$, which is false. 
	So $\alpha \in \dom(h_{v,\tau}) \setminus \dom(c_{r,\beta})$. 
	Hence, $\alpha \in \dom(C_\beta)$. 
	So $c_w(\beta,\alpha) = C_\beta(\alpha)$ and 
	$c_w(\beta,h_{v,\tau}(\alpha)) = C_\beta(h_{v,\tau}(\alpha))$. 
	By Claim (C), we have that 
	$$
	F_\tau(\min\{c_w(\beta,\alpha),c_w(\beta,h_{v,\tau}(\alpha))\}) < 
	\max\{c_w(\beta,\alpha),c_w(\beta,h_{v,\tau}(\alpha))\}.
	$$
	Now $v \res \tau$ forces (in $\p_\tau)$ 
	that $\dot f_\tau \res \zeta_1 = F_\tau$, and hence 
	so does $w \res \tau$. 
	Also, $w \res 0$ forces (in $\p_0$, and hence in $\p_\tau$) 
	that $\dot{\pi}_\beta \res \dom(c_{w,\beta}) = c_{w,\beta}$. 
	So 
	$$
	w \res \tau \Vdash_\tau 
	\dot f_\tau(\min\{\dot{\pi}_\beta(\alpha),\dot{\pi}_\beta(h_{w,\tau}(\alpha))\}) < 
	\max \{ \dot{\pi}_\beta(\alpha), \dot{\pi}_\beta(h_{w,\tau}(\alpha)) \}.
	$$
	which proves (i). 

	Now consider $\gamma \in \dom(h_{w,\tau}) \cap \beta$ which is different from $\alpha$. 
	First, assume that $\gamma \in \dom(h_{r,\tau})$. 
	Since $\beta \in I_{r,\tau}$, $h_{r,\tau}(\gamma) < \beta$. 
	Since $r \in E_\delta$, by Definition \ref{definition of E_delta}(2) 
	both $\gamma$ and $h_{r,\tau}(\gamma)$ are in $\dom(c_{r,\beta})$. 
	By statement (4) above, both $c_{r,\beta}(\gamma)$ and $c_{r,\beta}(h_{r,\tau}(\gamma))$ 
	are less than $\zeta$. 
	By statement (18c) above, 
	$F_\tau(c_{r,\beta}(\gamma))$ and $F_\tau(c_{r,\beta}(h_{r,\tau}(\gamma)))$ 
	are less than $\zeta_0$. 
	On the other hand, $\alpha$ and $h_{v,\tau}(\alpha)$ are in the domain of $C_\beta$ 
	and consequently $C_\beta(\alpha)$ and $C_\beta(h_{v,\tau}(\alpha))$ are both 
	greater than or equal to $\zeta_0$. 
	So $w \res \tau$ forces (in $\p_\tau$) 
	that 
	$$
	\dot{f}_\tau(\min \{ \dot{\pi}_\beta(h_{w,\tau}(\gamma)),\dot{\pi}_\beta(h_{w,\tau}(\alpha))\}) 
	= \dot{f}_\tau(\dot{\pi}_\beta(h_{w,\tau}(\gamma)))
	$$
	is strictly less than 
	$$
	\dot{\pi}_\beta(h_{w,\tau}(\alpha)) = 
	\max\{\dot{\pi}_\beta(h_{w,\tau}(\gamma)),\dot{\pi}_\beta(h_{w,\tau}(\alpha))\}.
	$$
	
	Secondly, assume that $\gamma \in \dom(h_{w,\tau}) \setminus \dom(h_{r,\tau})$. 
	Then $\gamma \in \dom(h_{v,\tau})$. 
	If $\gamma \in \dom(c_{r,\beta})$, then since $r \in E_\delta$ by 
	Definition \ref{definition of E_delta}(1) it follows that 
	$\gamma \in \dom(h_{r,\tau})$, which is false. 
	So $\gamma \in \dom(C_\beta)$. 
	By Claim (C), 
	$$
	F_\tau(\min \{ C_\beta(h_{v,\tau}(\gamma)), C_\beta(h_{v,\tau}(\alpha)) \}) < 
	\max \{ C_\beta(h_{v,\tau}(\gamma)), C_\beta(h_{v,\tau}(\alpha)) \}.
	$$
	So 
	$$
	w \res \tau \Vdash_\tau 
	\dot{f}_\tau(\min \{ \dot{\pi}_\beta(h_{w,\tau}(\gamma)), 
	\dot{\pi}_\beta(h_{w,\tau}(\alpha)) \}) < 
	\max \{ \dot{\pi}_\beta(h_{w,\tau}(\gamma)), 
	\dot{\pi}_\beta(h_{w,\tau}(\alpha)) \}.
	$$
	This completes the proof of Claim J.
	
	\bigskip
	
	Claims (D), (E), (F), and (J) confirm 
	that $w \le_\delta r$ and complete the proof of the theorem.
	\end{proof}

In light of the results of this article, 
a natural idea for proving the consistency of $(**)$ would be to interleave 
the fast club forcing into the iteration defined above. 
At first glance, this approach seems feasible 
due to the fast club forcing being $\omega_1$-closed and $2^\omega$-centered. 
Unfortunately, we are unable to prove that such an iteration preserves $\omega_2$.

\appendix
\section{A Family of Models for Side Conditions}

We now prove Theorem \ref{Ycal}. 
Fix a $\Box_{\omega_2}$-sequence 
$\vec c = \langle c_\alpha : \alpha < \omega_3, \ \alpha \ \text{limit} \rangle$. 
So each $c_\alpha$ is a club subset of $\alpha$ whose order type is at most $\omega_2$, 
and if $\beta \in \text{lim}(c_\alpha)$ then $c_\alpha \cap \beta = c_\beta$. 
A straightforward construction using the square sequence yields the following lemma. 
For a proof, see \cite[pp.\! 17--18]{mitchell}.

\begin{lemma}
	There exists a sequence 
	$\vec A = \langle A_{\eta,\xi} : \eta < \omega_3, \ \xi < \omega_2 \rangle$ 
	satisfying:
	\begin{enumerate}
	\item for all $\eta < \omega_3$, $\langle A_{\eta,\xi} : \xi < \omega_2 \rangle$ 
	is an increasing and continuous sequence of sets of size less than $\omega_2$ 
	with union equal to $\eta$;
	\item whenever $\beta \in \text{lim}(c_\alpha)$, 
	for all $\xi < \omega_2$, 
	$A_{\beta,\xi} = A_{\alpha,\xi} \cap \beta$.
\end{enumerate}
\end{lemma}

\begin{defn}
	For each $\beta < \omega_3$, for each $\xi < \omega_2$, 
	and for each $\gamma < \ot(A_{\beta,\xi})$, 
	define $a_{\beta,\xi,\gamma}$ to be the $\gamma$-th element of $A_{\beta,\xi}$.
\end{defn}

Note that whenever $\beta \in \lim(c_\alpha)$, 
the fact that $A_{\beta,\xi} = A_{\alpha,\xi} \cap \beta$ implies that for all 
$\gamma < \ot(A_{\beta,\xi})$, $a_{\beta,\xi,\gamma} = a_{\alpha,\xi,\gamma}$.

\begin{defn}
	For each $\beta < \omega_3$, define 
	$g_\beta : \omega_2 \times \omega_2 \to \beta$ by 
	\begin{itemize}
		\item $g_\beta(\xi,\gamma) = a_{\beta,\xi,\gamma}$ if 
		$\gamma < \ot(A_{\beta,\xi})$,
		\item $g_\beta(\xi,\gamma) = 0$ if 
		$\ot(A_{\beta,\xi}) \le \gamma$.
	\end{itemize}
\end{defn}

Since $2^{\omega_2} = \omega_3$, we can fix a bijection 
$f : \omega_3 \to H(\omega_3)$. 
Define $\Acal$ to be the structure 
$(H(\omega_3),\in,f,\vec c,\vec A)$. 
Due to $f$ being a part of the structure, $\Acal$ has definable Skolem functions. 
Observe that for all $\beta < \omega_3$, $g_\beta$ is definable in $\Acal$.

For simplicity in notation, for any $N \subseteq H(\omega_3)$, we abbreviate 
$\sup(N \cap \omega_3)$ as $\sup(N)$.

\begin{defn}
	Let $\Ycal$ denote the set of all $N \prec \Acal$ satisfying that 
	$|N| = \omega_1$ and $N^\omega \subseteq N$.
\end{defn}

Note that for all $M, N \in \Ycal$, $M \cap N \in \Ycal$.

\begin{lemma} \label{limits cofinal}
	For all $M \in \Ycal$, 
	$\text{lim}(c_{\sup(M)}) \cap M$ is cofinal in $\sup(M)$.
\end{lemma}

\begin{proof}
	This follows easily from the fact that $M^\omega \subseteq M$.
\end{proof}

\begin{lemma} \label{limit order type}
	Let $M \in \Ycal$. 
	If $\beta \in \text{lim}(M \cap \omega_3)$, then 
	$\ot(c_\beta) \le M \cap \omega_2$.
\end{lemma}

\begin{proof}
	First, assume that $\beta = \sup(M)$. 
	Then $\lim(c_\beta) \cap M$ is cofinal in $\beta$ by Lemma \ref{limits cofinal}. 
	By the coherence of the square sequence and elementarity, 
	this easily implies that 
	for cofinally many $\gamma \in c_\beta$, 
	$\ot(c_\beta \cap \gamma) \in M \cap \omega_2$, 
	which in turn implies that $\ot(c_\beta) \le M \cap \omega_2$. 
	Secondly, assume that $\beta < \sup(M)$. 
	Let $\alpha = \min((M \cap \omega_3) \setminus \beta)$. 
	Then by elementarity, $\beta$ is a limit point of $c_\alpha$, so 
	$c_\beta = c_\alpha \cap \beta$. 
	Hence, $\ot(c_\beta) \le \ot(c_\alpha) \in M \cap \omega_2$.
\end{proof}

\begin{defn}
	Define $\Ycal^+$ to be the set of all $N \in \Ycal$ such that 
	$\ot(c_{\sup(N)}) = N \cap \omega_2$.
\end{defn}

\begin{prop}
	The families $\Ycal$ and $\Ycal^+$ are stationary subsets of $[H(\omega_3)]^{\omega_1}$.
\end{prop}

\begin{proof}
	Since $\Ycal^+ \subseteq \Ycal$, it suffices to show that 
	$\Ycal^+$ is stationary. 
	Let $H : H(\omega_3)^{<\omega} \to H(\omega_3)$ be a given function. 
	Fix $L \prec \mathcal A$ which is closed under $H$ and 
	satisfies that $|L| = \omega_2$, $\omega_2 \subseteq L$, 
	$L^\omega \subseteq L$, and $L \cap \omega_3$ has cofinality $\omega_2$ 
	(using \textsf{CH}, it is easy to show that such a set exists, using an argument 
	similar to that in the next paragraph). 
	Let $\delta = L \cap \omega_3$. 
	Note that $c_\delta$ has order type $\omega_2$.
	
	Define by recursion an increasing and continuous sequence 
	$\langle M_\alpha : \alpha < \omega_1 \rangle$ 
	of subsets of $L$ of size $\omega_1$ as follows. 
	Let $M_0$ be an elementary substructure of $L$ 
	of size $\omega_1$ such that $\omega_1 \subseteq M_0$. 
	At limit stages take unions. 
	Suppose that $\alpha < \omega_1$ and $M_\alpha$ is defined and is an 
	elementary substructure of $L$ of size $\omega_1$. 
	Then $\sup(M_\alpha) < \delta$. 
	Since $\ot(c_\delta) = \omega_2$, we can 
	pick $\gamma_\alpha > \sup(M_\alpha)$ which is a limit point of $c_\delta$ 
	with cofinality $\omega_1$ satisfying that 
	$\ot(c_{\gamma_\alpha}) > M_\alpha \cap \omega_2$. 
	Now let $M_{\alpha+1}$ be the closure under the function $H$ and the definable 
	Skolem functions of $\mathcal A$ 
	of the set $M_\alpha \cup M_\alpha^\omega \cup \{ \gamma_\alpha, \ot(c_{\gamma_\alpha}) \}$. 
	Note that since $M_\alpha \subseteq L$ and $L^\omega \subseteq L$, 
	$M_{\alpha+1}$ is an elementary substructure of $L$ of size $\omega_1$. 
	This completes the construction.

	Let $M = \bigcup \{ M_\alpha : \alpha < \omega_1 \}$. 
	Then $M$ has size $\omega_1$, and by construction, 
	$M^\omega \subseteq M$, $M \prec \Acal$, and $M$ is closed under $H$. 
	So $M \in \Ycal$. 
	Let $\theta = \sup(M)$. 
	By construction, $\theta$ is a limit point of $c_\delta$, so 
	$c_\theta = c_\delta \cap \theta$. 
	We claim that $\ot(c_\theta) = M \cap \omega_2$, and hence $M \in \Ycal^+$. 
	By Lemma \ref{limit order type}, $\ot(c_\theta) \le M \cap \omega_2$. 
	On the other hand, for all $\beta < M \cap \omega_2$ there exists $\alpha < \omega_1$ 
	such that $\beta < M_\alpha \cap \omega_2 < 
	\ot(c_{\gamma_\alpha}) < M_{\alpha+1} \cap \omega_2 < M \cap \omega_2$. 
	But $c_{\gamma_\alpha}$ is a limit point of $c_\delta$ which is less than $\theta$, 
	and hence is a limit point 
	of $c_\delta \cap \theta = c_\theta$. 
	Therefore, $\beta < \ot(c_{\gamma_\alpha}) < \ot(c_\theta)$. 
	So $M \cap \omega_2 = \ot(c_\theta)$.
\end{proof}

\begin{lemma} \label{expressing members in Ycal}
	Let $M \in \Ycal$ and let $\alpha = \sup(M)$. 
	Then $M \cap \omega_3 = \{ g_\alpha(\xi,\gamma) : 
	\xi, \gamma \in M \cap \omega_2 \}$.
\end{lemma}

\begin{proof}
	($\subseteq$) Consider $\tau \in M \cap \omega_3$. 
	By Lemma \ref{limits cofinal}, 
	fix $\beta \in \lim(c_\alpha) \cap M$ greater than $\tau$ . 
	Then by elementarity, there is $\xi \in M \cap \omega_2$ 
	such that $\tau \in A_{\beta,\xi}$. 
	Let $\gamma = \ot(A_{\beta,\xi} \cap \tau)$. 
	Then $\gamma \in M$ and $\tau = a_{\beta,\xi,\gamma}$. 
	But $A_{\beta,\xi} = A_{\alpha,\xi} \cap \beta$, so 
	$\tau = a_{\beta,\xi,\gamma} = a_{\alpha,\xi,\gamma} = g_\alpha(\xi,\gamma)$. 

	($\supseteq$) 
	Suppose that $\xi, \gamma \in M \cap \omega_2$ and we show that 
	$g_{\alpha}(\xi,\gamma) \in M$. 
	If $\gamma \ge \ot(A_{\alpha,\xi})$, then $g_\alpha(\xi,\gamma) = 0 \in M$. 
	Assume that $\gamma < \ot(A_{\alpha,\xi})$. 
	Then $g_\alpha(\xi,\gamma) = a_{\alpha,\xi,\gamma} < \alpha$. 
	Since $\lim(c_\alpha) \cap M$ is cofinal in $\alpha$ by Lemma \ref{limits cofinal}, 
	we can fix 
	$\beta \in \lim(c_\alpha) \cap M$ which is greater than $a_{\alpha,\xi,\gamma}$. 
	Then $A_{\beta,\xi} = A_{\alpha,\xi} \cap \beta$. 
	So $a_{\alpha,\xi,\gamma} = a_{\beta,\xi,\gamma}$, and by 
	elementarity, $a_{\beta,\xi,\gamma} \in M$. 
	
\end{proof}

\begin{lemma} \label{g equal}
	Suppose that $M \in \Ycal$, $\beta \in \text{lim}(M \cap \omega_3)$, 
	$\beta < \sup(M)$, and 
	$\alpha = \min((M \cap \omega_3) \setminus \beta)$. 
	Then for all $\xi, \gamma < M \cap \omega_2$, 
	$g_\beta(\xi,\gamma) = g_\alpha(\xi,\gamma)$. 
\end{lemma}

\begin{proof}
	The statement is immediate if $\beta = \alpha$, so assume that $\beta < \alpha$. 
	Then an easy argument by elementarity shows that $\beta \in \text{lim}(c_\alpha)$. 
	Consequently, for all $\xi < \omega_2$, 
	$A_{\beta,\xi} = A_{\alpha,\xi} \cap \beta$, 
	and hence $\ot(A_{\beta,\xi}) \le \ot(A_{\alpha,\xi})$.  
	Now consider $\xi, \gamma < M \cap \omega_2$. 
	If $\gamma \ge \ot(A_{\alpha,\xi})$, then also $\gamma \ge \ot(A_{\beta,\xi})$, 
	so both $g_{\beta}(\xi,\gamma)$ and $g_\alpha(\xi,\gamma)$ are equal to $0$. 
	Suppose that $\gamma < \ot(A_{\alpha,\xi})$. 
	Since $\xi, \gamma \in M$, 
	$g_{\alpha}(\xi,\gamma) = a_{\alpha,\xi,\gamma} \in M \cap \alpha = M \cap \beta$. 
	So $a_{\alpha,\xi,\gamma} \in A_{\alpha,\xi} \cap \beta = A_{\beta,\xi}$. 
	As $A_{\beta,\xi}$ is an initial segment of $A_{\alpha,\xi}$, 
	$a_{\alpha,\xi,\gamma} = a_{\beta,\xi,\gamma}$, that is, 
	$g_{\alpha}(\xi,\gamma) = g_\beta(\xi,\gamma)$.
\end{proof}

\begin{prop}
	Let $M \in \Ycal$, $N \in \Ycal^+$, and suppose that 
	$M \cap \omega_2 < N \cap \omega_2$. 
	Then $M \cap N \in N$.
\end{prop}

\begin{proof}
	Since $M \cap N$ is an elementary substructure of $\Acal$, 
	$M \cap N = f[M \cap N \cap \omega_3]$. 
	So by the elementarity of $N$ in $\Acal$, 
	it suffices to show that $M \cap N \cap \omega_3 \in N$. 
	Let $\beta = \sup(M \cap N)$. 
	We claim that $N \setminus \beta$ is non-empty. 
	Otherwise, $\sup(N) = \beta$ and so 
	$\ot(c_{\beta}) = N \cap \omega_2$. 
	On the other hand, $\beta$ is a limit point of $M \cap \omega_3$ and hence 
	by Lemma \ref{limit order type}, 
	$\ot(c_\beta) \le M \cap \omega_2 < N \cap \omega_2$, which is a contradiction. 
	Let $\alpha = \min((N \cap \omega_3) \setminus \beta)$. 
	Since $M \cap N \in \Ycal$, by Lemma \ref{expressing members in Ycal} we have that 
	$M \cap N \cap \omega_3 = 
	\{ g_{\beta}(\xi,\gamma) : \xi, \gamma \in M \cap N \cap \omega_2 \}$. 
	By Lemma \ref{g equal} 
	and the fact that $M \cap N \cap \omega_2 = M \cap \omega_2$, it follows that 
	$M \cap N \cap \omega_3 = 
	\{ g_{\alpha}(\xi,\gamma) : \xi, \gamma \in M \cap \omega_2 \}$. 
	As $\alpha$ and $M \cap \omega_2$ are in $N$, so is $M \cap N \cap \omega_3$.
\end{proof}

This completes the proof of Theorem \ref{Ycal}.

\section*{Acknowledgements}

Part of this work was completed while the author was 
visiting Justin Moore at Cornell University 
in February 2026. 
The author thanks Justin Moore for funding his visit, for his hospitality, and for 
extensive discussions concerning the topics of this article. 
The author thanks Hannes Jakob and Jing Zhang for providing feedback on an earlier draft. 
Finally, the author thanks the anonymous referee for a careful reading of the manuscript and for many helpful comments and suggestions which improved the presentation.

\bigskip

\providecommand{\bysame}{\leavevmode\hbox to3em{\hrulefill}\thinspace}
\providecommand{\MR}{\relax\ifhmode\unskip\space\fi MR }
\providecommand{\MRhref}[2]{%
  \href{http://www.ams.org/mathscinet-getitem?mr=#1}{#2}
}
\providecommand{\href}[2]{#2}


\end{document}